\documentclass[11pt]{article}

\usepackage{a4wide}

\usepackage{amsmath,amsthm,amssymb}
\usepackage{mathtools}
\usepackage{bbm}
\usepackage{xcolor}
\usepackage{comment}

\theoremstyle{plain}

\DeclareMathOperator{\var}{\mathrm{Var}}

\newcommand{\eps}{\varepsilon}

\newtheorem{thm}{Theorem}[section]

\newtheorem{lem}[thm]{Lemma}
\newtheorem{prop}[thm]{Proposition}
\theoremstyle{definition}

\newtheorem*{rmk}{Remark}

\title{Benford's law and the C$\beta$E}
\author{Nedialko Bradinoff\footnote{Department of Mathematics, Royal Institute of Technology, Lindstedtsvägen 25, SE 10044, Stockholm Sweden. E-mail: nedialko@kth.se} \and Maurice Duits \footnote{Department of Mathematics, Royal Institute of Technology, Lindstedtsvägen 25, SE 10044, Stockholm Sweden. E-mail: duits@kth.se}}

\date{}

\begin{document}
\maketitle
\begin{abstract}
    We study the individual digits for the absolute value of the characteristic polynomial for the Circular $\beta$-Ensemble. We show that, in the large $N$ limit, the first digits obey Benford's Law  and the further digits become uniformly distributed. Key to the proofs is a bound on the rate of convergence in total variation norm  in the CLT for the logarithm of the absolute value of the characteristic polynomial. 
\end{abstract}

\setcounter{tocdepth}{1}
%\tableofcontents

\section{Introduction}

 In this paper we study the behavior of the individual  digits of the random function $|D_{N,\beta}(\theta)|$ where
\begin{equation} 
    D_{N, \beta}(\theta)= \prod_{j=1}^N (e^{i \theta}-e^{i \theta_j}),
\end{equation}
and the $\theta_j \in [0,2 \pi)$ are randomly chosen from the probability measure on $[0,2 \pi)^N$ proportional to
\begin{equation} \label{eq:betaEnsemble}
    \prod_{1 \leq j < k \leq N}|e^{i \theta_j}-e^{i \theta_k}|^\beta d\theta_1\cdots d \theta_N,
\end{equation}
with $\beta>0$ and $N \in \mathbb N$ large.
Note that \eqref{eq:betaEnsemble} is  the  joint probability distribution for the  eigenvalues of the Circular $\beta$-Ensemble (C$\beta$E) from random matrix theory and $D_{N,\beta}(\theta)$ is the characteristic polynomial. It is not difficult to see that the distribution is invariant under translations (modulo $2\pi$)  and thus the distribution of $|D_{N,\beta}(\theta)|$ is independent of $\theta$.  The C$\beta$E is a classical family of ensembles from random matrix theory.  For $\beta=1$, $\beta=2$ and $\beta=4$ the distribution \eqref{eq:betaEnsemble} coincides with that of  the eigenvalues of COE, CUE and CSE matrices \cite{Dyson}.  For general $\beta>0$ it concides with the eigenvalue distribution of randomly chosen truncated CMV matrices \cite{KillipNenciu}.   There is a vast amount of literature on the C$\beta$E and we do not intend to give an exhaustive  list of references. We will  only mention  the most relevant to the present paper.  For a general discussion and more background we refer to \cite[Chapter 2]{Forrester}.  

The universality principle dictates that, as $N\to \infty$, the fluctuations for the C$\beta$E are governed by objects that are  shared with many other point processes that have  similar interactions and the circular symmetry makes the C$\beta$E a particularly popular model for studying these fluctuations.  For instance, when zooming in at scales of order $\sim 1/n$ around a fixed point on the circle, the point process converges to the $\sin_\beta$-process. For $\beta=2$ this is a determinantal point process determined by the sine kernel. For general $\beta$ this process can be characterized as the spectrum of a random  operator \cite{KillipStoicu,Valko2,ValkoVirag}. Important  progress has also  been made on the limiting behavior, as $N \to \infty$, of $|D_{N,\beta}(\theta)|$ as a random function of $\theta$. The random field $\log |D_{N,\beta}(\theta)| $ converges to a Gaussian log-correlated field.  A first result in this direction is the CLT in \cite{Johansson} for linear statistics with smooth test functions (the proof is for $\beta=2$ but can be extended to general $\beta>0$). How such a CLT is related to  weak convergence of  $\log |D_{N,\beta}(\theta)| $ to a Gaussian log-correlated field was explained in \cite{Hughes}.  From this it is then natural to expect that an appropriately normalised $|D_{N,\beta} (\theta)|$ converges to a Gaussian Multiplicative Chaos measure. For $\beta=2$ this was proved in \cite{webb2,Webb}. For $\beta\neq 2$ this is still open and only proved for a gentle  regularization of  $|D_{N,\beta}(\theta)|$  in \cite{Gaultier}. For further  results on the relation between C$\beta$E and Gaussian Multiplicative Chaos  we refer to \cite{Bourgade1,ChaiibiNajnudel}.

The purpose of this paper is to discuss another universal behavior, namely we show that the leading digit satisfies Benford's law. This implies in particular that
\begin{equation}
    \lim_{N \to \infty} \mathbb P(\text{ first digit of } |D_{N,\beta}(\theta)| \text{ is } k)=\log_{10}\left(1+\frac{1}{k}\right).
\end{equation}
Thus, surprisingly,   the  leading digit of $|D_{N,\beta}(\theta)|$ is more likely to take a lower value than a higher one.   The base 10 is not special here and the same result holds for the leading digit in base $b$ with $\log_{10}$ replaced by $\log_{b}$. 

The  fascinating pattern of Benford's law is observed in variety of different data sets from statistics, physics and number theory.  It has various important   applications, such as detecting tax and election fraud. It was first observed by Newcomb in 1881, reportedly after noticing that in logarithmic tables the pages with lower leading digit were much more worn than the ones with higher leading digit.  Afterwards, it was rediscovered and generalized by Benford in 1938 \cite{Benford}, who examined a variety of sequences from both physics and mathematics. A simple and illustrative mathematical example is the sequence $a^n$ for $n \in \mathbb N$. The leading digit  will satisfy Benford's law in base $b$ if and only if $\log_b a \notin \mathbb Q$. Or, equivalently, if and only if addition by $\log_b a$ modulo 1 is ergodic implying that $n \log_b a\mod 1$  for $n=1,2,3,\ldots$ is uniformly distributed.
In fact, one expects Benford's law for any sequence $a_n$ to hold when $\log a_n$ is uniformly distributed $\mod 1$, see  \cite[Th. 1]{Diaconis}. 
For a beautiful overview of various quantities of Benford's law we refer to \cite{Berger,Diaconis} and the book  \cite{Miller}. 

\subsubsection*{Statement of main results}
We will prove the general version of Benford's law:

\begin{thm} \label{thm1}
  Let the digital expansion of $|D_{N,\beta}(\theta)|$ in base $b \in \mathbb N$  be given by 
    $$
        |D_{N,\beta}(\theta)|=\left(\sum_{k=1}^{\infty}d_{k,N}(\theta)b^{1-k}\right)b^{M},
    $$                                                                                
    where $M\in\mathbb{Z}$,  $d_{k,N}:[0,2\pi]^{N}\rightarrow\{0,1,...,b-1\}$, and  $d_{1,N}\not= 0.$
     Then, for any $\ell\in \mathbb Z$ and fixed $\theta\in [0,2 \pi)$, we have
    \begin{equation} \label{eq:Benford}
     \mathbb{P}( d_{1,N}(\theta)=k_1,d_{2,N}(\theta)=k_2,...,d_{\ell,N}(\theta)=k_\ell)=\log_{b}\left(1+\frac{1}{\sum_{j=1}^{\ell}k_j b^{\ell-j}}\right)+o \left((\log N)^{-3/2+\eps}\right),
    \end{equation}
    as $ N \rightarrow\infty$, for any $\eps>0$.
    \end{thm}

  The proof of this result will be given in Section \ref{proofsthms}.

  From \eqref{eq:Benford} we can compute the distribution of a single digit $d_{\ell,N}$, for $\ell\geq 2$,  to be: 
  \begin{equation} \label{eq:Benfordkthdigit}
    \mathbb{P}( d_{\ell,N}(\theta)=k)=\sum_{s=b^{\ell-2}}^{b^{\ell-1}-1} \log_{b}\left(1+\frac{1}{b s+k}\right)+o \left((\log N)^{-3/2+\eps}\right),
    \end{equation}
    as $N\to \infty$.
From \eqref{eq:Benfordkthdigit} we see that, for  any fixed $\ell$, and $N \to \infty$, the $\ell$-th digit in base $b$ is more likely to be close to $1$ than  $b-1$.    This effect, however, decreases when $\ell$ increases. In fact, if $\ell$ grows with $N$, one expects from \eqref{eq:Benfordkthdigit} that the distribution of the $\ell$-th digit will converge to a uniform distribution.  This brings us to our second main result of this paper. 

\begin{thm} \label{thm2}
   Let $L \in \mathbb N$ and, for $\ell=1,\ldots,L$, let $m_{\ell}:\mathbb{N}\rightarrow \mathbb{N}$ be a function such that $m_{\ell}(N)\rightarrow\infty$ as $N\rightarrow\infty$ and $m_i(N)<m_j(N)$ for $i<j$.  Then, for any $\eps>0$,
\begin{equation}
\mathbb{P}(d_{m_1(N),N}=k_1, \ldots, d_{m_L(N),N}=k_L)=\frac{1}{b^L}+\mathcal O\left(\frac{1}{(\log N)^{3/2-\varepsilon}}+ \frac{1}{b^{m_1(N)}}\right),
\end{equation}
as $N\to \infty$.
\end{thm}
The proof of this result will be given in Section \ref{proofsthms}.

\begin{rmk}
     For  the special value $\beta=2$, Benford's law for $|D_{N,\beta}(\theta)|$ as formulated in Theorem \ref{thm1} was first mentioned by Kontorovich and Miller \cite{KantMill}. However, we want to point out that their proof is not complete. Indeed, their argument is based on a heuristic derivation   by Keating and Snaith in \cite{KeatingSnaith} on the limiting behavior of the probability density function of $\log |D_{N,\beta}(\theta)|$ as $N\to \infty$ based on bounds on the cumulants. A rigorous justification of this heuristic derivation requires additional arguments (the strong bounds on the cumulants will only give convergence in distribution, but not convergence of the density function). In fact, our main technical arguments in   Section 3 below are   providing such  a justification. We will  return to this issue in Section \ref{sec:CLTweak} and also discuss a counterexample in Appendix \ref{sec:counter}. 
\end{rmk} 

\subsubsection*{Convergence in Total Variation norm}
We will now comment on our method for proving Theorems \ref{thm1} and \ref{thm2}.

Roughly speaking, a random variable $X$  will satisfy Benford's law if $\log X$ (we will restrict to positive valued random variables) is close to a uniform distribution on $\mathbb R$. The latter is of course not a well-defined probability measure and one typically arrives at Benford's law in a limiting procedure. For instance, if $X=X_n$ is a sequence of positive random variables such that $\var \log X_n \to \infty$ and, as $n\to \infty$, 
\begin{equation} \label{eq:logCLT}
 \frac{\log X_n -\mathbb  E \log X_n }{\sqrt{\var \log X_n }}\to \mathcal N(0,1),
\end{equation}
in distribution, then $\log_b X_n \mod 1$ is close to a uniform distribution on $[0,1]$, and this is an indication (but not a sufficient condition) that Benford's law may hold. From this point of view,  Benford's law can be seen as a Central Limit Theorem  for \emph{products} of independent random variables.  See for instance, \cite[Theorem 6.6]{Berger} for a result along this line for products of iid random variables. It is known \cite{KeatingSnaith,Su} that, for some explicit constant $c$,  $\var \log |D_{N,\beta}(\theta)| \sim  c \log N$ as $N\to \infty$ and that \eqref{eq:logCLT} holds with $X_N=|D_{N,\beta}(\theta)|$ which indicates that Theorem \ref{thm1} indeed holds. Note also  that one can write  $D_{N,\beta}(0)$, in law, as a product of independent random variables  (but not identically distributed), see for example \cite[Prop. 2.2]{BHNY}.

We stress however that \eqref{eq:logCLT} is not a sufficient requirement for Benford's law to hold, and it is not hard to construct counterexamples of sequences that satisfy \eqref{eq:logCLT} but not Benford's law. Indeed,  consider the random variable $X_n=10^{Y_n}$ where $Y_n$ is a binomial random variable $Y_n\sim Bin(n,\frac12)$. Then, clearly, $\log X_n$ satisfies \eqref{eq:logCLT} but the first digit of $X_n$ is always $1$ since $Y_n$ is discrete and takes values in $\mathbb N$. To obtain a continuous counterexample one can consider $X_n=10^{Y_n+\varepsilon Z}$ where $Z\sim \mathcal{N}(0,1)$. In this case, the first digit is not necessarily 1, but for small $\varepsilon$ it is very skewed towards $1$ and more so than in Benford's law.  The point is that (high) frequencies in the distribution function for $\log X_n$ can strongly influence the first digits. We therefore need a stronger control on the convergence in \eqref{eq:logCLT}. We prove the following criterion for Benford's law for a sequence of positive random variables in terms of convergence in total variation norm at the logarithmic scale.
\begin{thm}\label{thm:sufficientconditions}
Let $(\alpha_N)_{N\in\mathbb{N}},(\gamma_N)_{N\in\mathbb{N}}$ be an sequences of positive real numbers such that
$$\alpha_N\rightarrow\infty,\quad\text{as}\quad N\rightarrow\infty, \quad\gamma_{N}\rightarrow0,\quad\text{as}\quad N\rightarrow\infty.$$ 
Suppose $(X_N)_{N\in\mathbb{N}}$ is a sequence of positive random variables such that 
$$d_{TV}(\log X_N, \alpha_{N}\mathcal{N}(0,1))=\mathcal{O}(\gamma_N).$$
Assume further that $X_N$ has a digital expansion in base $b\in\mathbb{N}$, given by
    $$
       X_N=\left(\sum_{k=1}^{\infty}d_{k,N}b^{1-k}\right)b^{M}.
    $$ 
Then for any $\ell\in\mathbb{N}$
\begin{equation}\label{eqn:thm1eqn}
 \mathbb{P}( d_{1,N}=k_1,d_{2,N}=k_2,...,d_{\ell,N}=k_\ell)=\log_{b}\left(1+\frac{1}{\sum_{j=1}^{\ell}k_j b^{\ell-j}}\right)+\mathcal{O}\left(\gamma_N+\exp\left(-\frac{2\pi\alpha_N^2}{\log^2b}\right)\right),
\end{equation}
as $N\rightarrow\infty$. Furthermore, assuming the notation of Theorem \ref{thm2},
\begin{equation}\label{eqn:thm2eqn}
\left|\mathbb{P}(d_{m_1(N),N}=k_1, \ldots, d_{m_L(N),N}=k_L)-\frac{1}{b^L}\right|=\mathcal{O}\left(\gamma_N+\exp\left(-\frac{2\pi\alpha_N^2}{\log^2b}\right)+ \frac{1}{b^{m_1(N)}}\right),
\end{equation}
as $N\rightarrow\infty$.
\end{thm}

  The proof of this result will be given in Section \ref{proofsthms}.

We prove a bound on the total variation distance between $\log |D_{N,\beta}(\theta)|/\sqrt{\var \log |D_{N,\beta}(\theta)|}$ and $\mathcal{N}(0,1)$ that will imply Benford's law via Theorem \ref{thm:sufficientconditions}. The bound that we obtain for convergence in total variation is $o(1/(\log N)^{3/2-\varepsilon})$, for arbitrary $\varepsilon>0$, which we believe to be close to optimal. Note that it is consistent with previous bounds we found in the literature for convergence in Kolmogorov distance of $\mathcal O(1/(\log N)^{3/2})$, obtained in \cite{BHNY, Ashkan} for the classical case $\beta=2$. This bound also bounds the rate of convergence to Benford's law in Theorem \ref{thm1}.

Our analysis is based on a thorough analysis of the characteristic function, or Fourier transform, of the linear statistic $\log \left|D_{N,\beta}(\theta)\right|$, defined by
\begin{equation}
    \psi_{N,\beta}(t)=\mathbb E\left[\left|D_{N,\beta}(\theta)\right|^{it} \right]=\mathbb E\left[e^{it \log \left|D_{N,\beta}(\theta)\right|} \right].
\end{equation}
By the circular symmetry, the right-hand side does not depend on $\theta$ and we have therefore omitted $\theta$ in the notation on the left-hand side. What we show is a bound on $\psi_{N,\beta}(t)$ for all $t \in \mathbb R$ uniform in $N$.  Our method for bounding this function is based on the  Selberg integral which  gives a representation of $\psi_{N,\beta}(t)$ as a product of gamma functions. Note that bounds for $\psi_{N,\beta}(t)$ have been derived for $t$ in  intervals that can grow moderately with $N$ (see for instance in  \cite{DalBorgo}), but we will need bounds that are valid for arbitrary $t$ and give us bounds on the  $L^1$ and $L^2$ norms. These bounds produce bounds on the total variation distance in the Central Limit Theorem.  Note also that there are many results in the literature that give bounds on the tails of the distribution function for $\log |D_{N,\beta}(\theta)|$, but here we are interested in estimating the presence of higher frequencies. The total variation distance is the natural metric in this context.

We would like to point out that if one only wants to prove weak convergence to Benford's law in  Theorem  \ref{thm1} (without a rate of convergence), then there is an elegant argument by Giuliano \cite[Theorem 2.1]{Giuliano} saying that it is sufficient (and necessary)  to prove that $\psi_{N,\beta}(h) \to 0$ as $N\to \infty$  for non-zero integers $h \in \mathbb Z\setminus \{0\}$. We will however control $\psi_{N,\beta}(t)$ for all values of $t$, which allows us to include  a rate of convergence in Theorem \ref{thm1} and to prove Theorem \ref{thm2}.

\subsection*{Overview of the rest of the paper}

We discuss the Selberg integral and its implications for the cumulants of $\log \left|D_{N,\beta}(\theta)\right|$ in Section \ref{sec:selberg}. This section does not contain new results, but discusses  various known results from the literature. In Section \ref{sec:bound} we provide bounds on the characteristic function $\psi_{N,\beta}(t)$. These bounds allow us to prove Theorem \ref{thm1} and Theorem \ref{thm2} in Section \ref{proofsthms}. Finally, in Appendix \ref{sec:counter} we discuss an example showing that a weak CLT or bounds on the limiting behavior of the cumulants is not enough to deduce Benford's law.

\subsection*{Acknowledgements}
Both authors were  supported by the Swedish Research Council (VR),   grant no 2016-05450 and grant no. 2021-06015, the European Research Council (ERC), Grant Agreement No. 101002013.  We are grateful to Kurt Johansson and Gaultier Lambert for  discussions and pointing us to relevant references.

\section{The Selberg integral and some first implications} \label{sec:selberg}
We start with recalling some preliminaries on the Selberg integral and its implications for the C$\beta$E. All results in this section are well-known in the  random matrix theory literature, but for completeness we include various simple proofs.
\subsection{Preliminaries}
We will need the well-known gamma function $\Gamma:\{z\in\mathbb{C}: \Re(z)>0\}\rightarrow\mathbb{C}$, given by the absolute convergent integral,
$$
\Gamma(z)=\int_{0}^{\infty}x^{z-1}e^{-x}dx.
$$ Its extension to a meromorphic function on $\mathbb{C}\setminus\mathbb{Z}_{\leq0}$ has the following representation:
\begin{equation}\label{eqn:WGamma}
\Gamma(z)=\frac{e^{-\gamma z}}{z}\prod_{n=1}^{\infty}\left(1+\frac{z}{n}\right)^{-1}e^{z/n}.
\end{equation}
Selberg proved the following beautiful statement.
\begin{thm}[Selberg integral]\label{thm:Selberg}
Let $n\in\mathbb{N}$, $\alpha,\beta,\gamma\in\mathbb{C}$, with $Re(\alpha)>0,Re(\beta)>0$ and $Re(\gamma)>-\min\{(1/n),Re(\alpha)/(n-1), Re(\beta)/(n-1)\}$ and let
\begin{equation}
S_n(\alpha,\beta,\gamma)=\int_{0}^{1}...\int_{0}^{1}\prod_{j=1}^{n}t_j^{\alpha-1}(1-t_j)^{\beta-1}\prod_{1\leq j<k\leq n}|t_i-t_j|^{2\gamma}dt_1...dt_n.
\end{equation}
Then,
\begin{equation}\label{Selbergevaluated}
S_n(\alpha,\beta,\gamma)=\prod_{j=0}^{n-1}\frac{\Gamma(\alpha+j\gamma)\Gamma(\beta+j\gamma)\Gamma(1+(j+1)\gamma)}{\Gamma(\alpha+\beta+(n+j-1)\gamma)\Gamma(1+\gamma)}.
\end{equation}
\end{thm}
Selberg published his proof in Norwegian, \cite{Selberg}. To see the proof in English we refer the reader to \cite{Forrester}, and for an overview of the implications of the Selberg integral, \cite{ForresterWarnaar}. The following proposition is one such implication, which allows us to express the characteristic function of $\log|D_{N,\beta}(\theta)|$ in terms of the gamma function. The following observation was actually first published in \cite{Morris}.
\begin{prop}\label{prop:Mexpression}
Let
\begin{equation}
M_n(a,b,\gamma)=\frac{1}{(2\pi)^n}\int_{-\pi}^{\pi}...\int_{-\pi}^{\pi}\prod_{j=1}^{n}e^{\frac{1}{2}i\theta_j(a-b)}\left|1+e^{i\theta_j}\right|^{a+b}\prod_{1\leq j<k\leq n}\left|e^{i\theta_j}-e^{i\theta_k}\right|^{2\gamma}d\theta_1...d\theta_n,
\end{equation}
where $Re(a+b+1)>0$, $Re(\gamma)>-\min\{1/n,Re(a+b+1)/(n-1)\}$. Then,
\begin{equation}
M_n(a,b,\gamma)=\prod_{j=0}^{n-1}\frac{\Gamma(1+a+b+j\gamma)\Gamma(1+(j+1)\gamma)}{\Gamma(1+a+j\gamma)\Gamma(1+b+j\gamma)\Gamma(1+\gamma)}.
\end{equation}
\end{prop}
\begin{proof}
See the discussion in Section 3.9 in \cite{Forrester}, or \cite{Morris}.
\end{proof}
Having introduced the last proposition, we show how it implies various identities for the characteristic function of $\log\left|D_{N,\beta}(\theta)\right|.$
\subsection{Implications of the Selberg integral}
Using the Selberg integral we can compute the characteristic function of $\log\left|D_{N,\beta}(\theta)\right|$ in terms of the gamma function. 
\begin{lem}\label{lemma:charfngammafn}
The characteristic function of $\log\left|D_{N,\beta}(\theta)\right|$ with respect to the Circular $\beta$-Ensemble is
\begin{equation}\label{eqn:charfn}
\psi_{N,\beta}(t)=\prod_{j=0}^{N-1}\frac{\Gamma\left(1+it+j\frac{\beta}{2}\right)\Gamma\left(1+j\frac{\beta}{2}\right)}{\Gamma\left(1+it/2+j\frac{\beta}{2}\right)^2}.
\end{equation}
\end{lem}

\begin{proof}
The characteristic function of $\log|D_{N,\beta}(\theta)|$ is
\begin{multline}
\psi_{N,\beta}(t)=\mathbb{E}\left[e^{it\log\left|D_{N,\beta}(\theta)\right|}\right]=\mathbb{E}\left[\left|D_{N,\beta}(\theta)\right|^{it}\right]\\
=\frac{1}{Z_N}\int_{0}^{2\pi}...\int_{0}^{2\pi}\prod_{j=1}^N \left|e^{i\theta_j}-e^{i\theta}\right|^{it}\left|e^{i\theta_j}-e^{i\theta_k}\right|^\beta\prod_{j=1}^{N}d\theta_j,
\end{multline}
where $Z_N$ is a normalising constant.
Because of the rotation invariance of the measure, there is no dependence on $\theta$ of the above integral. We can thus set $\theta=\pi$ and obtain
\begin{equation}
\psi_{N,\beta}(t)=\int_{-\pi}^{\pi}...\int_{-\pi}^{\pi}\frac{1}{(2\pi)^NC_N(\beta/2)}\prod_{j=1}^N \left|1+e^{i\theta_j}\right|^{it}\prod_{1\leq j<k\leq N}\left|e^{i\theta_j}-e^{i\theta_k}\right|^\beta\prod_{j=1}^{N}d\theta_j.
\end{equation}
We use $M_N(it/2,it/2,\beta/2)$ to evaluate the integral, and $M_{N}(0,0,\beta/2)$ to obtain the normalizing constant. Combining the two results, we obtain:
$$\psi_{N,\beta}(t)=\frac{\Gamma\left(1+\frac{\beta}{2}\right)^N}{\Gamma\left(1+N\frac{\beta}{2}\right)}\prod_{j=0}^{N-1}\frac{\Gamma\left(1+it+j\frac{\beta}{2}\right)\Gamma\left(1+(j+1)\frac{\beta}{2}\right)}{\Gamma\left(1+\frac{it}{2}+j\frac{\beta}{2}\right)^2\Gamma\left(1+\frac{\beta}{2}\right)},$$
and this gives us the result, when rearranged.
\end{proof}
Lemma \ref{lemma:charfngammafn} can be found in \cite{Su}. There is an alternative proof without relying on the Selberg integral \cite{Bourgade,BHNY}. Indeed, the joint probability distribution function  \eqref{eq:betaEnsemble}   can be obtained as the eigenvalue distribution of a unitary matrix constructed out of random Verblunsky coefficients \cite{KillipNenciu}. Based on this identification one can write $D_{N,\beta}(0)$ as the product of independent (but not identically distributed!) random variables. This also explains the product structure on the right-hand side of \eqref{eqn:charfn}.

 The following straightforward corollary of Lemma \ref{lemma:charfngammafn} yields an expression for the characteristic function, which is suited for the analysis of the characteristic function for $|t|>\sqrt{\log N}$, that is beyond the standard deviation.

\begin{lem}\label{lemma:family}
Let $\psi_{N,\beta}$ be the characteristic function of $Y_N=\log|D_{N,\beta}(\theta)|$ w.r.t. the Circular $\beta$-Ensemble, then:
\begin{equation}\label{eqn:family}
\psi_{N,\beta}(t)=\prod_{n=1}^{\infty}\prod_{j=0}^{N-1}\frac{(n+j\frac{\beta}{2}+\frac{i}{2}t)^2}{(n+j\frac{\beta}{2}+it)(n+j\frac{\beta}{2})}=\frac{\Gamma(1+it)}{\Gamma(1+it/2)^2}\prod_{n=1}^{\infty}\prod_{j=1}^{N-1}\frac{\left(1+\frac{i}{2}\frac{t}{n+\frac{j\beta}{2}}\right)^2}{1+i\frac{t}{n+\frac{j\beta}{2}}}.
\end{equation}
\end{lem}

\begin{proof}
We apply \eqref{eqn:WGamma} to the right-hand side of \eqref{eqn:charfn}. We recall that $\Gamma(z+1)=z\Gamma(z)$ together with \eqref{eqn:WGamma} to deduce that 
$$
    \frac{\Gamma(1+it+j\delta)\Gamma(1+j\delta)}{\Gamma(1+\frac{it}{2}+j\delta)^2}=\prod_{n=1}^{\infty}\frac{(n+it/2+j\delta)^2}{(n+it+j\delta)(n+j\delta)}$$ 
for $\delta\geq0$. The result follows.
\end{proof}
Lemma \ref{lemma:family} is stated in \cite{DalBorgo} in their discussion about rational $\beta$.  It gives rise to infinite series expressions for the cumulants, which we will use in the next subsection. We also use it to bound concretely every term in the product in order to give an upper bound for $|\psi_{N,\beta}(t)|$ for $|t|<N^{5/7}\frac{\beta}{2}$.

\subsection{Computing the cumulants and a CLT} \label{sec:CLTweak}
We proceed to compute the cumulants of  $\log |D_{N,\beta}(\theta)|$ using Lemma \ref{lemma:family}, analogous to  \cite{KeatingSnaith} in the case of the CUE. We also bound the absolute value of the third derivative of $\log \psi_{N,\beta}(t)$. We show that these results give rise to weak convergence of $\frac{\log|D_{N,\beta}(\theta)|}{\sqrt{\frac{1}{\beta}\log(N)}}$ to $\mathcal{N}(0,1)$. 
\begin{lem}\label{lemma:cumulantcalculation}
Let $D_{N,\beta}(\theta)$ be the characteristic polynomial of the C${\beta}$E. Then:
$$\mathbb{E}\left(\log |D_{N,\beta}(\theta)|\right)=0,$$ 
\begin{equation} \label{eqn:varianceseries}
\var\left(\log|D_{N,\beta}(\theta)|\right)=\frac{1}{2}\sum_{j=0}^{N-1}\sum_{n=1}^{\infty}\frac{1}{\left(n+\frac{j\beta}{2}\right)^2},
\end{equation}
and for $k\geq 3$, the $k$-th  cumulant $C_k^{(N)}$,
\begin{equation}\label{eqn:cumulants}
C_k^{(N)}=(-1)^k\frac{2^{k-1}-1}{2^{k-1}}(k-1)!\sum_{j=0}^{N-1}\sum_{n=1}^{\infty}\frac{1}{\left(n+j\frac{\beta}{2}\right)^{k}}.
\end{equation}
Furthermore,
\begin{equation}\label{eqn:near0}
\left|\frac{\partial^3}{\partial t^3}\log(\psi_{N,\beta}(t))\right|\leq\frac{5}{2}\sum_{j=0}^{N-1}\sum_{n=1}^{\infty}\frac{1}{\left(n+\frac{j\beta}{2}\right)^3}<\frac{5}{\beta}+\frac{15}{4}, \forall t.
\end{equation}
\end{lem}

\begin{rmk}
The expression for the cumulants, \eqref{eqn:cumulants}, can also be rewritten in terms of special functions such as the polygamma function, $\psi^{(m)}(z)=\frac{d^{m+1}}{dz^{m+1}}\log\Gamma(z)$, or the Hurwitz zeta function, $\zeta(l,s)=\sum_{n=1}^{\infty}\frac{1}{(n+s)^l}$,
$$C_k^{(N)}=\frac{2^{k-1}-1}{2^{k-1}}\sum_{j=0}^{N-1}\psi^{(k-1)}\left(1+j\frac{\beta}{2}\right)=(-1)^k\frac{2^{k-1}-1}{2^{k-1}}(k-1)!\sum_{j=0}^{N-1}\zeta\left(k,j\frac{\beta}{2}\right),$$
the first equality giving an analogue of the expression that Keating and Snaith have in \cite{KeatingSnaith}.
\end{rmk}
\begin{proof}
Using \eqref{eqn:charfn}, we see that 
$$
    \log(\psi_{N,\beta}(t))
        =\sum_{j=0}^{N-1}
   \left(\log\Gamma\left(1+it+j\frac{\beta}{2}\right)+\log\Gamma\left(1+j\frac{\beta}{2}\right)-2\log\Gamma\left(1+i \frac{t}{2}+j\frac{\beta}{2}\right)\right).$$
   We can then write the derivatives of the cumulant generating function in terms of the polygamma function  $\psi^{(m)}(z)=\frac{d^{m+1}}{dz^{m+1}}\log\Gamma(z)$, which is meromorphic  with poles at $\mathbb{Z}_{\leq 0}$,
$$   
\frac{\partial^k}{\partial t^k}\log(\psi_{N,\beta})(t)
=\sum_{j=0}^{N-1}i^k\left(\psi^{(k-1)}\left(1+it+j\frac{\beta}{2}\right)-\left(\frac{1}{2}\right)^{k-1}\psi^{(k-1)}\left(1+i\frac{t}{2}+j\frac{\beta}{2}\right)\right).
$$
Utilising the  expansion of the polygamma function,
$$\psi^{(m)}(z)=(-1)^{m+1}m!\sum_{k=0}^{\infty}\frac{1}{(z+k)^{m+1}}\text{ for } m\geq 1,$$ 
we can rewrite the $k$-th derivative of the cumulant generating function as
\begin{equation}\label{eqn:kthcumulantderivative}
\frac{\partial^k}{\partial t^k}\log(\psi_{N,\beta})(t)=(k-1)!\sum_{j=0}^{N-1}\sum_{n=1}^{\infty}\left(\frac{-i}{n+\frac{j\beta}{2}}\right)^k\left(\frac{1}{\left(1+i\frac{t}{n+\frac{j\beta}{2}}\right)^k}-\frac{1}{2^{k-1}}\frac{1}{\left(1+\frac{i}{2}\frac{t}{n+\frac{j\beta}{2}}\right)^k}\right). 
\end{equation}
Evaluating the first and second derivative at $0$ we obtain the expectation and variance. The higher derivatives evaluated at $0$ give the higher cumulants. To deduce \eqref{eqn:near0} we bound the expression for the third derivative from \eqref{eqn:kthcumulantderivative},
\begin{multline*}
    \left|\frac{\partial^3}{\partial t^3}\log(\psi_{N,\beta})(t)\right|\leq 2\sum_{n=1}^{\infty}\sum_{j=0}^{N-1}\left(\frac{1}{n+\frac{j\beta}{2}}\right)^3\left|\frac{1}{\left(1+i\frac{t}{n+\frac{j\beta}{2}}\right)^3}-\frac{1}{2^{2}}\frac{1}{\left(1+\frac{i}{2}\frac{t}{n+\frac{j\beta}{2}}\right)^3}\right|\\
    \leq 2\sum_{n=1}^{\infty}\sum_{j=0}^{N-1}\left(\frac{1}{n+\frac{j\beta}{2}}\right)^3\frac{5}{4}=\frac{5}{2}\sum_{n=1}^{\infty}\frac{1}{n^3}+\frac{5}{2}\sum_{n=1}^{\infty}\sum_{j=1}^{N-1}\left(\frac{1}{n+\frac{j\beta}{2}}\right)^3.
\end{multline*}
Since the function $x\mapsto\frac{1}{\left(n+\frac{x\beta}{2}\right)^3}$ is decreasing we estimate the second sum by an integral as follows:
$$
\left|\frac{\partial^3}{\partial t^3}\log(\psi_{N,\beta})(t)\right|\leq\frac{5}{2}\sum_{n=1}^{\infty}\frac{1}{n^3}+\frac{5}{2}\sum_{n=1}^{\infty}\sum_{j=1}^{\infty}\left(\frac{1}{n+\frac{j\beta}{2}}\right)^3\leq\frac{5}{2}\zeta(3)+\frac{5}{2}\sum_{n=1}^{\infty}\int_{0}^{\infty}\frac{dx}{\left(n+\frac{x\beta}{2}\right)^3}.$$
We apply then a few elementary estimates 
$$\left|\frac{\partial^3}{\partial t^3}\log(\psi_{N,\beta})(t)\right|\leq \frac{15}{4}+\frac{5}{\beta}\sum_{n=1}^{\infty}\frac{1}{2n^2}\leq\frac{15}{4}+\frac{5}{\beta}\frac{\pi^2}{12}\leq\frac{15}{4}+\frac{5}{\beta},$$
and this completes the proof.
\end{proof}

In particular, Lemma \ref{lemma:cumulantcalculation} shows that the variance of $\log \left|D_{N,\beta}(\theta)\right|$ grows with $N$, while its higher cumulants remain bounded, as in $\cite{KeatingSnaith}$. This implies  a rate of convergence for $\psi_{N,\beta}(t/T_N)$ converging to $e^{-t^2/2}$ when $t=\mathcal O\left(\sqrt{\var(\log \left|D_{N,\beta}(\theta)\right|)}\right)$ as $N \to \infty$.  See also \cite{Ashkan}.
\begin{lem}\label{lemma:near0} 
Let $T_N=\sqrt{\var(\log|D_{N,\beta}(\theta)|)}$, then:
\begin{enumerate}
\item
\begin{equation}\label{eqn:variancebound}
\frac{1}{\beta}\log\left(1+\frac{\beta N}{2}\right)<\var\left(\log\left|D_{N,\beta}(\theta)\right|\right)\leq \frac{1}{\beta}\left(\log(N)+1+\beta \right).
\end{equation}
\item
 for $|t|<T_N$,
\begin{equation}\label{eqn:near0diff}
\left|\psi_{N,\beta}(t/T_N)-e^{-t^2/2}\right|\leq \frac{c_1}{T_N^3}e^{-t^2/2}|t^3|,
\end{equation}
where $c_1= \frac{c}{6}\exp\left(\frac{c}{6}\right)$, $c=\frac{5}{\beta}+\frac{15}{4}$.
\end{enumerate}
\end{lem}
\begin{proof}
\begin{enumerate}
\item
We estimate the series for the variance, \eqref{eqn:varianceseries}, first from above and then below:
$$
\frac{1}{2}\sum_{j=0}^{N-1}\sum_{n=1}^{\infty}\frac{1}{(n+\frac{j\beta}{2})^2}=\pi^2/12+\frac{1}{2}\sum_{j=1}^{N-1}\sum_{n=1}^{\infty}\frac{1}{(n+\frac{j\beta}{2})^2}\leq\pi^2/12+\frac{1}{2}\sum_{j=1}^{N-1}\int_{0}^{\infty}\frac{dx}{(x+\frac{j\beta}{2})^2}
$$
$$
=\pi^2/12+\frac{1}{2}\sum_{j=1}^{N-1}\frac{1}{\frac{j\beta}{2}}<\frac{1}{\beta}+\frac{\pi^2}{12}+\frac{1}{\beta}\log(N-1)\leq \frac{1}{\beta}\left(\log(N)+1+\beta \right).
$$
Similarly, we use the monotonicity of $\frac{1}{x+\frac{j\beta}{2}}$, $\frac{1}{(x+\frac{j\beta}{2})^2}$ and a dissection, 
$$
\frac{1}{2}\sum_{j=0}^{N-1}\sum_{n=1}^{\infty}\frac{1}{(n+\frac{j\beta}{2})^2}\geq\frac{1}{2}\sum_{j=0}^{N-1}\frac{1}{1+\frac{j\beta}{2}}=\frac{1}{\beta}\log\left(1+\frac{\beta N}{2}\right).
$$

\item 
Taylor expanding $\log\psi_{N,\beta}(t)$ around $0$, it follows that for some $\xi(t)\in[0,t/T_N]$,
$$\log\psi_{N,\beta}(t/T_N)=\frac{-t^2}{2}+\frac{1}{3!}\left(\frac{t}{T_N}\right)^3\frac{\partial^3}{\partial t^3}\log(\psi_{N,\beta})(\xi).$$
Then we can estimate the difference, $|\psi_{N,\beta}(t)-e^{-t^2/2}|$, using the upper bound on $\frac{\partial^3}{\partial t^3}\log(\psi_{N,\beta})$, \eqref{eqn:near0}:
$$\left|\exp(\log(\psi_{N}(t/T_N)))-e^{-t^2/2}\right|=e^{-t^2/2}\left|\exp\left(\frac{t^3}{6T_N^3}\frac{\partial^3}{\partial t^3}\log(\psi_{N,\beta})(\xi)\right)-1\right|$$
$$\leq e^{-t^2/2} \frac{ct^3}{6T_N^3}\exp\left(\frac{ct^3}{6T_N^3}\right).$$
Assuming $|t|\leq T_N$, we deduce \eqref{eqn:near0diff}.
\end{enumerate}
\end{proof}
\begin{rmk}
A result such as \eqref{eqn:near0diff} can be established purely from \eqref{eqn:cumulants} for $|t|<\frac{1}{2}T_N$. 
\end{rmk}
From this lemma we see that, as $N \to \infty$
$$
 \frac{\log |D_{N,\beta}(\theta)|-\mathbb E \log |D_{N,\beta}(\theta)|}{\sqrt{\var \log |D_{N,\beta}(\theta)|}} \to \mathcal{N}(0,1),
$$
in distribution. This is a well-known result for $\beta=2$ \cite{KeatingSnaith}, but for general $\beta>0$ first proved in \cite{Su}. However, this is not enough to establish Benford's law for the leading digits and the uniform law for the further digits. We will prove that the convergence holds even in total variation and we give an estimate on the rate of convergence. This will allow us to prove Theorem \ref{thm1} and \ref{thm2} simultaneously.

\section{Bounding the characteristic function} \label{sec:bound}
A hurdle to purely working with the cumulants \eqref{eqn:cumulants} is that the cumulant generating function of $\log\left|D_{N,\beta}(\theta)\right|$ has a radius of convergence $1$. This means that the cumulants only allow us to control $\psi_{N,\beta}(t)$ for $|t|\leq 1$. In this section we will prove an upper bound on the characteristic function $\psi_{N,\beta}(t)$ that is valid for all $t\in \mathbb R$ uniformly in $N$.  This will allow us to prove convergence in total variation for the CLT for $\log |D_{N,\beta}(\theta)|$. We will need to distinguish two regimes:
\begin{enumerate}
    \item The low frequency regime: $|t|\leq N^{5/7} \beta/2$.
    \item The high frequency regime: $|t|>N^{5/7} \beta/2$.
\end{enumerate} 
In the first region, for $|t|<\frac{\beta}{8}N^{\frac{1}{6}}$, the asymptotic behavior of $\psi_{N,\beta}(t)$ is understood very well, with the use of Binet's identity for the gamma function and the Abel-Plana summation formula, see Theorem 5.1 and Theorem 4.14 in \cite{DalBorgo}. We note that our bound in this regime, given in Proposition \ref{proposition:smallt} is valid over a larger region, and relies on the expression in Lemma \ref{lemma:family} and explicit estimates. \subsection{The low frequency regime: $|t|\leq N^{5/7}\frac{\beta}{2}$}
In the low frequency regime we prove the following bound:
\begin{prop}[Estimate for small $t$]\label{proposition:smallt}
Let $\psi_{N,\beta}$ be the characteristic function of the Circular $\beta-$Ensemble. Then if  $|t|\leq N^{5/7}\frac{\beta}{2},$ $N$ sufficiently large,
\begin{equation}\label{eqn:propsmallt}
|\psi_{N,\beta}(t/T_N)|\leq \exp\left(-c_2\frac{t^2}{2}\right),
\end{equation}
where $c_2= \frac{1}{32}.$
\end{prop}
\begin{proof}
We recall expression  \eqref{eqn:family} for the characteristic function given in Lemma \ref{lemma:family} and make the observation that
$$\psi_{N,\beta}(t)=\psi_{J,\beta}(t)\prod_{j=J}^{N-1}\frac{\left(1+\frac{i}{2}\frac{t}{n+\frac{j\beta}{2}}\right)^2}{1+i\frac{t}{n+\frac{j\beta}{2}}}.$$
Hence, taking absolute value and squaring on both sides we see that, for $J\in\mathbb{N}$, $J<N$,
\begin{equation}\label{eqn:smallttoproceed}
|\psi_{N,\beta}(t)|^2=|\psi_{J,\beta}(t)|^2\prod_{j=J}^{N-1}\prod_{n=1}^{\infty}\frac{\left(1+\frac{1}{4}\left(\frac{t}{n+\frac{j\beta}{2}}\right)^2\right)^2}{\left(1+\left(\frac{t}{n+\frac{j\beta}{2}}\right)^2\right)}.
\end{equation}
As $|\psi_{N,\beta}(t)|$ is even in $t$, we restrict our attention to $t\geq0$. We denote $\xi(t,n,j,\beta)=\frac{t}{n+\frac{j\beta}{2}}$. 
Observe that
\begin{equation}\label{eqn:smalltexpbound}
\frac{\left(1+\frac{\xi^2}{4}\right)^2}{1+\xi^2}=\frac{1+\frac{\xi^2}{2}+\frac{\xi^4}{16}}{1+\xi^2}\leq\left(1-\frac{7}{32}\xi^2\right)\leq\exp\left(-\frac{7}{32}\xi^2\right), \text{ when $|\xi|<1.$}
\end{equation}
For $t\leq 1+\frac{J\beta}{2}$, we apply \eqref{eqn:smalltexpbound} to \eqref{eqn:smallttoproceed}, to see 
\begin{equation*}\label{eqn:firstestsmallt}
|\psi_{N,\beta}(t)|^2\leq|\psi_{J,\beta}(t)|^2\exp\left(-\sum_{j=J}^{N-1}\sum_{n=1}^{\infty}\frac{7t^2}{32}\frac{1}{(n+\frac{j\beta}{2})^2}\right), \text{ for $|t|\leq\frac{J\beta}{2}+1$}.
\end{equation*}
Now, for fixed $j\in\{J,J+1,\ldots,N-1\}$,
$$\sum_{n=1}^{\infty}\frac{1}{(n+\frac{j\beta}{2})^2}\geq \int_{1}^{\infty}\frac{dx}{(x+\frac{j\beta}{2})^2}=\frac{1}{1+\frac{j\beta}{2}}\quad\text{and}\quad\sum_{j=J}^{N-1}\frac{1}{1+\frac{j\beta}{2}}\geq\int_{J}^{N}\frac{dx}{1+x\beta/2}=\frac{2}{\beta}\log\left(\frac{1+\frac{N\beta}{2}}{1+\frac{J\beta}{2}}\right).$$
Thus, collecting these computations and recalling that $|\psi_{J,\beta}(t)|\leq 1$, as a characteristic function, we see
\begin{equation}\label{eqn:smalltdeduction}
|\psi_{N,\beta}(t)|^2\leq \exp\left(-\frac{7t^2}{16\beta}\log\left(\frac{1+\frac{N\beta}{2}}{1+\frac{J\beta}{2}}\right)\right), \text{ for $|t|\leq\frac{J\beta}{2}+1$}.
\end{equation}
Now any fixed $t$ with $0<t<\frac{N\beta}{2}$ satisfies \eqref{eqn:smalltdeduction} with $J=J(t)=\left\lceil\frac{2}{\beta}(t-1)\right\rceil$. Hence using that $J=\left\lceil\frac{2}{\beta}(t-1)\right\rceil\leq \frac{2}{\beta}(t-1)+1$ in \eqref{eqn:smalltdeduction},
$$|\psi_{N,\beta}(t)|^2\leq \exp\left(-\frac{7t^2}{16\beta}\left(\log\left(1+\frac{N\beta}{2}\right)-\log\left({t+\frac{\beta}{2}}\right)\right)\right).$$
Now assume $t<\frac{\beta}{2}N^{\alpha}$, and fix $1>\alpha'>\alpha>0$ and N sufficiently large so that $N^{\alpha'}>N^\alpha+1$. It follows then that
$$|\psi_{N,\beta}(t)|^2\leq \exp\left(-\frac{7t^2}{16\beta}\log\left(\frac{N}{N^{\alpha}+1}\right)\right)\leq \exp\left(\frac{-7(1-\alpha')t^2}{16\beta}\log(N)\right).$$ 
Hence, recalling \eqref{eqn:variancebound}, and that $T_N^2=\var(\log\left|D_{N,\beta}(\theta)\right|)$, and assuming $N>e^{1+\beta}$ so that $1+\beta<\log N$,
$$|\psi_{N,\beta}(t/T_N)|^2\leq \exp\left(\frac{-7(1-\alpha')t^2}{32}\right).$$
Choosing $\alpha=5/7, \alpha'=6/7$, and $N$ sufficiently large ($N\geq 6$) we conclude the result.
\end{proof}
The bound given in \eqref{eqn:propsmallt} is already sufficient to prove convergence in Kolmogorov distance of $\frac{\log \left|D_{N,\beta}(\theta)\right|}{T_N}$ to $\mathcal{N}(0,1)$ at rate $\mathcal{O}\left(\frac{1}{T_N^3}\right)$. This result is proved for   $\beta=2$ in \cite{Ashkan}  (and other classical compact groups), using Mod-$\phi$ convergence. We mention an alternative proof in the appendix. For the proof of our results and especially Theorem \ref{thm2}, we need to control all frequencies and obtain convergence in total variation.
 
The proof of Proposition \ref{proposition:smallt} relies on the term-wise estimate \eqref{eqn:smalltexpbound}, which would give us the gaussian decay we want only if $|t|\leq const N^{\alpha}$ for $\alpha<1$. We need a different way to estimate when $t$ is outside of that region.  

\subsection{High frequency regime: $|t|\geq N^{\alpha}\frac{\beta}{2}$}
To prove total variation convergence we still need find an appropriate bound for $\psi_{N,\beta}(t)$ for $|t|>N^{5/7}\frac{\beta}{2}$.   The following rough estimate suffices.
\begin{prop}\label{prop:allt}
Let $\psi_{N,\beta}$ be the characteristic function of the Circular $\beta$-Ensemble. Then there is a constant $C=C(\beta)>0$ such that
\begin{equation}
|\psi_{N,\beta}(t)|\leq \min\left(1,\frac{C}{t^2}\right).
\end{equation}
\end{prop}

The above result relies on the following lemma.
\begin{lem}\label{lemma:usingstirling}
For $\alpha\geq 0$ define
\begin{equation*}
\phi_{\alpha}(t)=\frac{\Gamma(1+\alpha+it)\Gamma(1+\alpha)}{\Gamma\left(1+\alpha+\frac{it}{2}\right)^2}.
\end{equation*}
\begin{enumerate}
\item 
Then
\begin{equation}\label{eqn:verblunski}
 |\phi_{\alpha}|\leq 1,
 \end{equation}
\item and for $c_*= \frac{9e}{4\sqrt 2\pi}$ and $0\leq\varepsilon<1$ ,
\begin{equation}\label{eqn:stirlingextralemma}
|\phi_{\varepsilon}(t)|\leq \frac{c_*}{|t|^{1/2}}\quad\text{for}\quad |t|>4e.
\end{equation}
\end{enumerate}
\end{lem}

\begin{proof}
\begin{enumerate}
\item
One way to see this is by observing that $|\phi_{\alpha}|$ has a global maximum at $0$. Alternatively, an elegant way to prove this statement is by noting that $\phi_{\alpha}$ is the characteristic function of a suitably constructed random variable, see lemma 3.1 in \cite{Bourgade}.
\item
The result is a straightforward application of a Stirling approximation for the Gamma function, see (2.1.1) in \cite{stirling}.
\end{enumerate}
\end{proof}
\begin{comment}
\begin{lem}\label{lemma:verblunski}
For $\alpha>0$, consider the probability density on the unit disc, $\mathbb{D}$,
$$\sim \left(1-|z|^2\right)^{\alpha-1}.$$
Then for $t\in{\mathbb{R}}$,
\begin{equation}
\phi_{\alpha}(t)=\mathbb{E}\left[|1-z|^{it}\right]=\frac{\Gamma(\alpha+1+it)\Gamma(1+\alpha)}{\Gamma\left(\alpha+1+\frac{it}{2}\right)^2},
\end{equation}
is the characteristic function of a random variable.
\end{lem}
\begin{proof}
See lemma 3.1 in \cite{Bourgade}.
\end{proof}
\end{comment}
Now the proof of Proposition \ref{prop:allt} is straightforward. 
\begin{proof}[Proof of Proposition \ref{prop:allt}]
Assuming the notation of Lemma \ref{lemma:usingstirling}, since $\phi_{\alpha}(t)$ is a characteristic function for any $\alpha\geq0$, $|\phi_{\alpha}(t)|\leq 1$.
Moreover, using that $\Gamma(z+1)=z\Gamma(z)$ we see that
$$\phi_{\alpha}(t)=\prod_{k=1}^{[\alpha]}\frac{(\alpha+1-k+it)(1+\alpha-k)}{\left(1+\alpha-k+\frac{it}{2}\right)^2}\phi_{\alpha-[\alpha]}(t),$$
and in particular since $|\phi_{\alpha-[\alpha]}(t)|\leq1$ by Lemma \ref{lemma:usingstirling}, \eqref{eqn:verblunski}, it follows that
$$|\phi_{\alpha}(t)|\leq (1+[\alpha])!\left(\frac{4}{|t|}\right)^{[\alpha]}.$$
Hence making use of the product structure in \eqref{eqn:charfn} we see that
\begin{equation}\label{eqn:fastbound}
|\psi_{N,\beta}(t)|=\prod_{j=0}^{N-1}|\phi_{\alpha_j}(t)|\leq\left|\phi_{\alpha_j}(t)\right|\leq \left(1+[\alpha_j]\right)!\left(\frac{4}{|t|}\right)^{[\alpha_j]},\text{ where }\quad\alpha_{j}=j\frac{\beta}{2}.
\end{equation}
The result follows from the right-hand side of \eqref{eqn:fastbound} when  we fix $ j =\left[\frac{4}{\beta}\right]+1.$
\end{proof}

\subsection{ Tail behaviour of $\psi_{N,\beta}(t)$ } \label{refinedestimate}

Proposition \ref{prop:allt} gives a bound for the high frequency regime that is sufficient for our purposes. We briefly investigate a more refined estimate of the tail of the characteristic function in Section \ref{refinedestimate}. The discussion is not needed for our main results on the distribution of the digits, but since the bound is stronger and the proof is in the same spirit as the proof of Proposition \ref{proposition:smallt} we include it here.

 \begin{prop}\label{prop:bigt}
Let $\psi_{N,\beta}$ be the characteristic function of $\log\left|D_{N,\beta}(\theta)\right|$. Then, for $t\geq N\beta\sqrt2 >4e$, 
\begin{equation}\label{eqn:propbigt}
|\psi_{N,\beta}(t)|\leq \frac{c_*^{N}}{|t|^{N/2}}\exp\left(-\frac{N^2\beta}{50}\right)\exp\left(c_3 N\right),
\end{equation}
where $c_*$ is the same as in Lemma \ref{lemma:usingstirling} and $c_3= \frac{\beta/8-1}{2N}+\frac{3}{8}\beta+\frac{3}{2}$. 
\end{prop}

\begin{proof}[Proof of Proposition \ref{prop:bigt}]
By an iterative application of the identity $\Gamma(z+1)=z\Gamma(z)$ we can rewrite \eqref{eqn:charfn} in the form,
\begin{equation}\label{eqn:formbigt}
\psi_{N,\beta}(t)=\frac{\Gamma(1+it)}{\Gamma(1+it/2)^2}\prod_{j=1}^{N-1}\frac{ \Gamma\left(1+it+\frac{j\beta}{2}-\left[\frac{j\beta}{2}\right]\right)\Gamma\left(1+\frac{j\beta}{2}-\left[\frac{j\beta}{2}\right]\right)}{\Gamma\left(1+\frac{it}{2}+\frac{j\beta}{2}-\left[\frac{j\beta}{2}\right]\right)^2}P_{N,\beta}(t),
\end{equation}
where
\begin{equation*}\label{eqn:Pdefinition}
P_{N,\beta}(t)=\prod_{j=1}^{N-1}\prod_{k=1}^{\left[\frac{j\beta}{2}\right]}\frac{\left(1+it+\frac{j\beta}{2}-k\right)\left(1+\frac{j\beta}{2}-k\right)}{\left(1+\frac{it}{2}+\frac{j\beta}{2}-k\right)^2}=\prod_{j=1}^{N-1}\prod_{k=1}^{\left[\frac{j\beta}{2}\right]}\frac{\left(1+\frac{it}{1+\frac{j\beta}{2}-k}\right)}{\left(1+\frac{it/2}{1+\frac{j\beta}{2}-k}\right)^2}.
\end{equation*}
To bound the terms involving the gamma function in \eqref{eqn:formbigt}, we use Lemma \ref{lemma:usingstirling}. Hence,
$$
|\psi_{N,\beta}(t)|^2\leq\frac{c_*^{2N}}{|t|^{N}} |P_{N,\beta}(t)|^2,\quad\text{where}\quad|P_{N,\beta}(t)|^2=\prod_{j=1}^{N-1}\prod_{k=1}^{\left[\frac{j\beta}{2}\right]}\frac{\left(1+\frac{t^2}{(1+\frac{j\beta}{2}-k)^2}\right)}{\left(1+\frac{t^2}{4(1+\frac{j\beta}{2}-k)^2}\right)^2}.
$$
To obtain \eqref{eqn:propbigt} it remains to bound $|P_{N,\beta}(t)|$ from above. By invoking that $x\leq\exp(x-1)$ for $x>0$ termwise, we observe that for fixed $j\in\{1,2,...,N-1\}$,

$$\prod_{k=1}^{\left[\frac{j\beta}{2}\right]}\frac{\left(1+\frac{t^2}{(1+\frac{j\beta}{2}-k)^2}\right)}{\left(1+\frac{t^2}{4(1+\frac{j\beta}{2}-k)^2}\right)^2}\leq\exp\left( \sum_{k=1}^{\left[\frac{j\beta}{2}\right]}\left(S_{1,j}-S_{2,j}\right)\right),$$
where
\begin{equation*}
S_{1,j}=\sum_{k=1}^{\left[\frac{j\beta}{2}\right]}\frac{t^2}{2}\frac{1}{\frac{t^2}{4}+(1+\frac{j\beta}{2}-k)^2}\quad\text{and}\quad  S_{2,j}=\sum_{k=1}^{\left[\frac{j\beta}{2}\right]}\frac{3t^4}{16}\frac{1}{\left(\frac{t^2}{4}+(1+\frac{j\beta}{2}-k)^2\right)^2}.
\end{equation*}
In order to give an upper bound on the right-hand side we proceed to bound $S_{1,j}$ from above and $S_{2,j}$ from below. The function $x\mapsto \frac{1}{\frac{t^2}{4}+(1+\frac{j\beta}{2}-x)^2}$ is increasing in $x$ on $[0, \frac{j\beta}{2}+1)$, hence 
\begin{multline}\label{eqn:s1j}
S_{1,j}\leq \frac{t^2}{2} \int_{1}^{\left[\frac{j\beta}{2}\right]+1}\frac{dx}{\frac{t^2}{4}+(x-\frac{j\beta}{2}-1)^2}=\frac{t^2}{2}\int_{-\frac{j\beta}{2}}^{\left[\frac{j\beta}{2}\right]-\frac{j\beta}{2}}\frac{dy}{\frac{t^2}{4}+y^2}
\\= t\arctan\left(\frac{j\beta}{t}\right)-t\arctan\left(\frac{2}{t}\left(\frac{j\beta}{2}-\left[\frac{j\beta}{2}\right]\right)\right).
\end{multline}
We proceed in a similar way with $S_{2,j}$. %First we rewrite the expression as
%$$S_{2,j}=\sum_{k=0}^{\left[\frac{j\beta}{2}\right]-1}\frac{3t^4}{16}\frac{1}{\left(\frac{t^2}{4}+(1+\frac{j\beta}{2}-\left[\frac{j\beta}{2}\right]+k)^2\right)^2}.$$
Utilising that the function $x\mapsto\frac{1}{\left(\frac{t^2}{4}+(1+\frac{j\beta}{2}-\left[\frac{j\beta}{2}\right]+x)^2\right)^2}$ is decreasing on  $[0,\infty)$, we use a dissection estimate to see

\begin{multline}\label{eqn:s2j}
S_{2,j}=\sum_{k=0}^{\left[\frac{j\beta}{2}\right]-1}\frac{3t^4}{16}\frac{1}{\left(\frac{t^2}{4}+(1+\frac{j\beta}{2}-\left[\frac{j\beta}{2}\right]+k)^2\right)^2}\geq\frac{3t^4}{16}\int_{0}^{\left[\frac{j\beta}{2}\right]-1}\frac{dx}{\left(\frac{t^2}{4}+(1+\frac{j\beta}{2}-\left[\frac{j\beta}{2}\right]+x)^2\right)^2}
\\=\frac{3}{4}t\left(\arctan\left(\frac{j\beta}{t}\right)-\arctan\left(\frac{2}{t}\left(1+\frac{j\beta}{2}-\left[\frac{j\beta}{2}\right]\right)\right)\right)
\\+\frac{3}{8}t^2\left(\frac{(\frac{j\beta}{2})}{(\frac{j\beta}{2})^2+\frac{t^2}{4}}-\frac{1+\frac{j\beta}{2}-\left[\frac{j\beta}{2}\right]}{(1+\frac{j\beta}{2}-\left[\frac{j\beta}{2}\right])^2+\frac{t^2}{4}}\right).
\end{multline}
Hence, combining the two estimates, \eqref{eqn:s1j}, \eqref{eqn:s2j}, we see that
\begin{equation}\label{eqn:differencebound}
S_{1,j}-S_{2,j}\leq \frac{t}{4}\arctan\left(\frac{j\beta}{t}\right)-\frac{3}{8}t^2\frac{\frac{j\beta}{2}}{(\frac{j\beta}{2})^2+\frac{t^2}{4}}+3,
\end{equation}
where the constant comes from bounding $\frac{3}{4}t\arctan\left(\frac{2}{t}\left(1+\frac{j\beta}{2}-\left[\frac{j\beta}{2}\right]\right)\right)<3.$ Proceeding to sum \eqref{eqn:differencebound} over $j$, we see
\begin{equation}\label{eqn:targetboundbigt}
|P_{N,\beta}(t)|^2\leq \exp\left(\sum_{j=1}^{N-1}(S_{1,j}-S_{2,j})\right)\leq \exp\left(\sum_{j=1}^{N-1}\left(\frac{t}{4}\arctan\left(\frac{j\beta}{t}\right)-\frac{3}{8}t^2\frac{\frac{j\beta}{2}}{(\frac{j\beta}{2})^2+\frac{t^2}{4}}+3\right)\right).
\end{equation}
First we bound from above the positive part of the sum on the right-hand side of \eqref{eqn:targetboundbigt},
\begin{equation}\label{eqn:positiveterms}
\frac{t}{4}\sum_{j=1}^{N-1}\arctan\left(\frac{j\beta}{t}\right) \leq\frac{t}{4}\int_{1}^{N}\arctan\left(\frac{x\beta}{t}\right)dx%=\frac{t^2}{4\beta}\int_{\beta/t}^{N\beta/t}\arctan ydy
%\\=\frac{t^2}{4}\left(y\arctan y-\frac{\log(y^2+1)}{2}\right)_{\beta/t}^{N\beta/t}
%\\=\frac{t^2}{4\beta}\left(\frac{N\beta}{t}\arctan\left(\frac{N\beta}{t}\right)-\frac{\log\left(\left(\frac{N\beta}{t}\right)^2+1\right)}{2}\right)-\frac{t^2}{4\beta}\left(\frac{\beta}{t}\arctan\left(\frac{\beta}{t}\right)-\frac{\log\left(\left(\frac{\beta}{t}\right)^2+1\right)}{2}\right)
\leq\frac{Nt}{4}\arctan\left(\frac{N\beta}{t}\right)-\frac{t^2}{8\beta}\log\left(\left(\frac{N\beta}{t}\right)^2+1\right)+\frac{\beta}{8},
\end{equation}
where the constant term comes from bounding $\log\left(\left(\frac{\beta}{t}\right)^2+1\right)<\frac{\beta^2}{t^2},$ and $ \frac{\beta}{t}\arctan(\frac{\beta}{t})>0$. Then we bound the negative part of the sum in \eqref{eqn:targetboundbigt} from below. %that is $\frac{3}{8}t^2\sum_{j=1}^{N-1}\frac{\frac{j\beta}{2}}{(\frac{j\beta}{2})^2+\frac{t^2}{4}}$. %Note that the function $x\mapsto \frac{x}{x^2+\frac{t^2}{4}}$ is increasing on $[0,t/2)$ and decreasing on $[t/2,\infty)$, hence we need to bound with a little more care. For $\frac{N\beta}{2}\leq t<N\beta$,
%\begin{multline}\label{eqn:sumintest}\frac{3}{8}t^2\sum_{j=1}^{N-1}\frac{\frac{j\beta}{2}}{(\frac{j\beta}{2})^2+\frac{t^2}{4}}= \frac{3}{8}t^2\sum_{j=1}^{[t/\beta]}\frac{\frac{j\beta}{2}}{(\frac{j\beta}{2})^2+\frac{t^2}{4}}+\frac{3}{8}t^2\sum_{j=[t/\beta]+1}^{N-1}\frac{\frac{j\beta}{2}}{(\frac{j\beta}{2})^2+\frac{t^2}{4}}\\
%\geq \frac{3}{8}t^2\int_{0}^{[t/\beta]}\frac{x\frac{\beta}{2}dx}{(x\frac{\beta}{2})^2+\frac{t^2}{4}}+\frac{3}{8}t^2\int_{[t/\beta]+1}^{N}\frac{
%\left(\frac{x\beta}{2}\right)dx}{\left(\frac{x\beta}{2}\right)^2+\frac{t^2}{4}}\geq \frac{3}{8}t^2\int_{0}^{N}\frac{\left(\frac{x\beta}{2}\right)dx}{\left(\frac{x\beta}{2}\right)^2+\frac{t^2}{4}}-\frac{3}{8}t^2\frac{t/2}{\frac{t^2}{4}+\frac{t^2}{4}},\end{multline}
%where the last inequality holds because $\frac{x}{x^2+\frac{t^2}{4}}$ has a maximum at $x=\frac{t}{2}$. Making the substitution $y=\left(\frac{x\beta}{2}\right)^2+\frac{t^2}{4}$, and simplifying the second term, we see that the right-hand side of \eqref{eqn:sumintest} equals:
%\begin{equation}
%\frac{3}{8\beta}t^2\int_{\frac{t^2}{4}}^{(\frac{N\beta}{2})^2+\frac{t^2}{4}}\frac{dy}{y}-\frac{3}{8}t\geq \frac{3}{8\beta}t^2\log\left(\frac{(\frac{N\beta}{2})^2+\frac{t^2}{4}}{\frac{t^2}{4}}\right)-\frac{3}{8}N\beta. 
%\end{equation}
Note that the function $x\mapsto \frac{x}{x^2+\frac{t^2}{4}}$ is increasing on $[0,t/2)$, and hence since $t\geq N\beta\sqrt 2$, and so $\frac{j\beta}{2}\in\left[0,t/2\right)$ for $j\in\{1,2,\ldots, N\}$ it follows that
\begin{equation}\label{eqn:negativeterms}
\frac{3}{8}t^2\sum_{j=1}^{N-1}\frac{\frac{j\beta}{2}}{(\frac{j\beta}{2})^2+\frac{t^2}{4}}\geq \int_{0}^N \frac{\left(\frac{x\beta}{2}\right)dx}{\left(\frac{x\beta}{2}\right)^2+\frac{t^2}{4}}- \frac{3t^2}{8}\frac{\frac{N\beta}{2}}{(\frac{N\beta}{2})^2+\frac{t^2}{4}} 
%\\=\frac{3}{8\beta}t^2\log\left(\frac{(\frac{N\beta}{2})^2+\frac{t^2}{4}}{\frac{t^2}{4}}\right)-\frac{3t^2}{8}\frac{\frac{N\beta}{2}}{(\frac{N\beta}{2})^2+\frac{t^2}{4}}
%\\\geq\frac{3}{8\beta}t^2\log\left(\frac{(\frac{N\beta}{2})^2+\frac{t^2}{4}}{\frac{t^2}{4}}\right)-\frac{3}{4}N\beta
 \geq\frac{3}{8\beta}t^2\log\left(\left(\frac{N\beta}{t}\right)^2+1\right)-\frac{3}{4}N\beta.
\end{equation}
%Hence, combining the two cases, and simplifying the expression within the logarithm, we see that
%\begin{equation}
%\frac{3}{8}t^2\sum_{j=1}^{N-1}\frac{\frac{j\beta}{2}}{(\frac{j\beta}{2})^2+\frac{t^2}{4}} \geq\frac{3}{8\beta}t^2\log\left(\left(\frac{N\beta}{t}\right)^2+1\right)-\frac{3}{4}N\beta.
%\end{equation}
Plugging in the results from \eqref{eqn:positiveterms} and \eqref{eqn:negativeterms} into \eqref{eqn:targetboundbigt}, we see that
$$|P_{N,\beta}(t)|^2\leq \exp\left(\frac{Nt}{4}\arctan\left(\frac{N\beta}{t}\right)-\frac{t^2}{2\beta}\log\left(\left(\frac{N\beta}{t}\right)^2+1\right)+\frac{\beta}{8}+\frac{3}{4}N\beta+3(N-1)\right).$$
Finally, observing that
\begin{equation*}
\log(4y^2+1)-y\arctan(2y) \geq \begin{cases}\frac{16}{25}y^2 \text{ for $y\in[0,1/4]$},
\\\frac{1}{25} \text{ for $y\in(1/4,1]$},
\end{cases}
\end{equation*}
we deduce by setting $y=\frac{N\beta}{2[t]}$ that
$$
|P_{N}(t)|\leq\begin{cases} \exp\left(-\frac{t^2}{4\beta}\frac{16}{25}\left(\frac{N\beta}{2t}\right)^2\right)\exp(Nc_3)=\exp\left(-\frac{N^2\beta}{25}\right)\exp\left(c_3N\right) \text{for $t\geq2N\beta$},
\\ \exp\left(-\frac{1}{100\beta}t^2\right)\exp(Nc_3)\leq \exp\left(-\frac{N^2\beta}{50}\right)\exp\left(c_3 N\right) \text{ for $\sqrt2 N\beta \leq t< 2N\beta$}.
\end{cases}
$$
The claimed estimate, \eqref{eqn:propbigt}, follows.
\end{proof}
\subsection{Convergence of $\psi_{N,\beta}(t)$ in $L^1(\mathbb{R})$ and $L^2(\mathbb{R})$}
The following result summarizes the conclusions of this section.
\begin{thm}\label{thm:boundingsectionsummary}
Let $\psi_{N,\beta}(t)$ be the characteristic function of $\log D_{N,\beta}(\theta),$ and let  \\$T_N=\sqrt{\var \log\left|D_{N,\beta}(\theta)\right|}$. Then the following estimates hold, where $c_1, c_2, c_3>0$ are (small) constants,
    \begin{equation}\label{eqn:combineestimates}
    \left|\psi_{N,\beta}(t/T_N)-e^{-t^2/2}\right|\leq\begin{cases}
    \frac{c_1}{T_N^3}\exp\left(-\frac{t^2}{2}\right)|t|^{3} \text{ for $|t|\leq T_N$},
    \\ e^{-t^2/2}+\exp\left(-c_2\frac{t^2}{2}\right) \text{ for $T_N\leq|t|<\frac{\beta}{2}N^{5/7}T_N$},
    \\ e^{-t^2/2}+c_3\left(\frac{T_n}{|t|}\right)^{2} \text{ for $|t|> \frac{\beta}{2}N^{5/7}T_N$}.
    \end{cases}
    \end{equation}
Furthermore,    
\begin{equation}\label{eqn:L1convergence}
\left\|\psi_{N,\beta}(t/T_N)-e^{-t^2/2}\right\|_{L^1}=\mathcal{O}\left(T_N^{-3}\right),
\end{equation}
and
\begin{equation}\label{eqn:L2convergence}
\left\|\psi_{N,\beta}(t/T_N)-e^{-t^2/2}\right\|_{L^2}=\mathcal{O}\left(T_N^{-3}\right).
\end{equation}
\end{thm}
\begin{proof}
The first line on the right-hand side of \eqref{eqn:combineestimates} follows from Lemma \ref{lemma:near0}, the second line from Proposition \ref{proposition:smallt}, and the third line from Proposition \ref{prop:bigt}. Then the bounds on the $L^1$ and $L^2$ norm follow from integrating the right-hand side of \eqref{eqn:combineestimates}, correspondingly in $\int_{\mathbb{R}}\left|\psi_{N,\beta}(t/T_N)-e^{-t^2/2}\right|^2dt\quad$ and $\quad\int_{\mathbb{R}}\left|\psi_{N,\beta}(t/T_N)-e^{-t^2/2}\right|dt$,
\begin{multline*}
\left\|\psi_{N,\beta}(t/T_N)-e^{-t^2/2}\right\|_{L^2}^2\\=\mathcal{O}\left(\frac{1}{T_N^6}\int_{-T_N}^{T_N}e^{-t^2}|t|^6dt+\int_{T_N}^{\frac{\beta}{2}N^{5/7}T_N}e^{-Ct^2}dt+\int_{\frac{\beta}{2}N^{5/7}T_N}^{\infty}e^{-t^2}dt+\int_{\frac{\beta}{2}N^{5/7}T_N}^{\infty}\frac{T_N^4}{t^4}dt\right)
\\=\mathcal{O}\left(\frac{1}{T_N^6}+\frac{1}{T_N}\exp\left(-CT_N^2\right)+\frac{T_N^4}{T_N^3N^{15/7}}
\right)=\mathcal{O}\left(T_N^{-6}\right),
\end{multline*}
and analogously,
\begin{equation*}
\left\|\psi_{N,\beta}(t/T_N)-e^{-t^2/2}\right\|_{L^1}=\mathcal{O}\left(T_N^{-3}\right).
\end{equation*}
The proof is complete.
\end{proof}

We are now ready to prove a CLT for $\log\left|D_{N,\beta}(\theta)\right|$.

\subsection{Bounding the total variation distance}
The above bounds on the convergence in $L^1$ norm and $L^2$ norm allow us to prove a CLT for $\log \left|D_{N,\beta}(\theta)\right|$ not only holds in distribution, but we even have convergence in total variation. Recall that for $P$ and $Q$ probability measures on $\mathbb{R}$, the total variation distance between $P$ and $Q$ is given by
\begin{equation*}
d_{TV}(P,Q)=\sup_{U\in\mathcal{B}(\mathbb{R})}|P(U)-Q(U)|.
\end{equation*}

\begin{thm}\label{thm:totalvariation}(Total variation convergence)
  Let $D_{N,\beta}(\theta)$ be the characteristic polynomial of the C$\beta$E and let $T_N=\sqrt{\var \left(\log|D_{N,\beta}(\theta)|\right)}$. Then:
    \begin{equation}
d_{TV}\left(\log\left|D_{N,\beta}(\theta)\right|, \mathcal{N}\left(0,T_N^2\right)\right)=o\left(\frac{1}{T_N^{3-\varepsilon}}\right),
    \end{equation}
as $N\to \infty$, for any $\eps>0$.
    \end{thm}

    \begin{proof}
    The following argument is standard in this context. Denote the indicator function of the set $A$ by $\chi_A$. Let $\rho_N(x)$ be the density of $Z_N=\frac{\log|D_{N,\beta}(\theta)|}{T_N}$ and fix $U\in\mathcal{B}(\mathbb{R})$. Then
    \begin{multline}
    \left|\mathbb{P}\left(Z_N\in U\right)-\mathbb{P}\left(\mathcal{N}(0,1)\in U\right)\right|=\left|\int_{\mathbb{R}}\chi_{U}\left(\rho_{N}(x)-\frac{1}{\sqrt{2\pi}}e^{-x^2/2}\right)dx\right|\\
    \leq \left|\int_{\mathbb{R}}\chi_{U\cap(-M,M)}\left(\rho_{N}(x)-\frac{1}{\sqrt{2\pi}}e^{-x^2/2}\right)dx\right| +\mathbb{P}\left(|Z_N|\geq M\right)+ P\left(|\mathcal{N}(0,1)|\geq M\right)\\
    \leq  \frac{\sqrt{2M}}{2\pi}\left\|\psi_{N,\beta}(t/T_N)-e^{-t^2/2}\right\|_{L^2}+\mathbb{P}\left(|Z_N|\geq M\right)+ P\left(|\mathcal{N}(0,1)|\geq M\right),
    \end{multline}
    by an application of Cauchy-Schwarz and Plancherel's theorem.  The latter two summands satisfy
    $$P\left(|\mathcal{N}(0,1)|\geq M\right)\leq \frac{2e^{-M^2/2}}{M\sqrt{2\pi}} \quad\text{and}\quad \mathbb{P}\left(|Z_N|\geq M\right)\leq\frac{1}{M^{2k}} \mathbb{E}\left[Z_N^{2k}\right]=\mathcal O\left(\frac{1}{M^{2k}}\right),$$
 where the bound on the $2k$-th moment of $Z_N$ is deduced from the boundedness of the higher cumulants in Lemma \ref{lemma:cumulantcalculation}.  Hence by Theorem \ref{thm:boundingsectionsummary}, we can conclude,
$$ \left|\mathbb{P}\left(Z_N\in U\right)-\mathbb{P}\left(\mathcal{N}(0,1)\in U\right)\right|=\mathcal{O}\left(\frac{M}{T_N^3}+M^{-2k}\right)\quad\text{as}\quad N\rightarrow\infty,\quad\text{for arbitrary $k\in\mathbb{N}$}.$$
  By choosing $M=T_N^{\varepsilon}$, for any fixed $\varepsilon>0$ and taking $k$ sufficiently large one concludes the result.
    \end{proof}
    
 \begin{rmk}
A careful inspection of our analysis of the characteristic function in fact reveals that we can allow for varying $\beta$ with $N$ and will still have convergence in total variation as long as $\beta= o\left(\log N\right)$ (that is, as long as the variance grows) i.e. we can allow for heating and moderate freezing with $N$. However for the sake of clarity we keep $\beta$ fixed in this work.
\end{rmk}
The rate of convergence in Theorem \ref{thm:totalvariation} of $o\left({T_N^{-3+\varepsilon}}\right)$ should be compared to  the rate of convergence in Kolmogorov distance established in \cite{BHNY, Ashkan} for $\beta=2$, that is $\mathcal O\left({T_N^{-3}}\right)$. In fact with some extra work our analysis produces a lower bound for the total variation distance that is $\Omega\left({T_N^{-3}}\right)$.

Having proved Theorem \ref{thm:totalvariation} we are now ready to prove our main results. 

\section{Proofs of Theorems \ref{thm1}, \ref{thm2} and \ref{thm:sufficientconditions}}
In this section we prove Theorems \ref{thm1}, \ref{thm2} and \ref{thm:sufficientconditions}.  
\subsection{Preliminaries on the distribution of individual digits}
We begin this section with a discussion on the connection between CLTs in total variation at the logarithmic scale with growing variance and results regarding the distribution of leading digits, leading up to Theorem \ref{thm:sufficientconditions}. For the results that follow we introduce some notation.   For $m\in\mathbb{N}$,  $k_1\in\{1,2,\ldots,b-1\}$ and $k_2,\ldots,k_{m}\in\{0,1,\ldots,b-1\}$, denote
$$I_{k_1,k_2,\ldots,k_{m}}=\left[\log_{b}\left(\sum_{j=1}^{m} k_jb^{1-j}\right), \log_{b}\left( \sum_{j=1}^{m} k_jb^{1-j}+b^{1-m}\right)\right)\subset [0,1).$$
The next simple lemma motivates this definition.
%\begin{defn}
%For $x>0$, $k\in\mathbb{N}$, denote $x[k]_B$ to be the $k$'th digit of $x$ base $B>0$.
%\end{defn}
%\begin{rmk}
%We just write $x[k]$ when the base is implicitly clear from the context.
%\end{rmk}

\begin{lem}\label{lemma:intervalsdescription}
Let $x\in\mathbb{R}_{>0}$, with digital expansion base $b$
    $$
     x=\left(\sum_{j=1}^{\infty}d_{j}b^{1-j}\right)b^{M}, \quad M\in\mathbb{Z}.
    $$
    Let $(k_n)_{n\in\mathbb{N}}\in\left\{0,1,2,\ldots, b-1\right\}^{\mathbb{N}}$ be values of digits base $b$ with $k=k_1\not=0$, and let $(m_n)_{n\in\mathbb{N}}$ be a strictly increasing sequence of natural numbers. Let $L,\ell\in\mathbb{N}$ be natural numbers with $L,\ell\geq 2$. Then:
\begin{enumerate}
\item 
\begin{equation} \label{eqn:1digitcase}
\left\{d_1=k\right\}=\left\{\log_{b}x\in\bigcup_{n\in\mathbb{Z}}\left(I_k+n\right)\right\},
\end{equation}

\item
\begin{equation} \label{eqn:manyfixeddigits}
\left\{d_{1}=k_1, d_{2}=k_2,...,d_{\ell}=k_{\ell}\right\}=
\left\{\log_{b}x\in\bigcup_{n\in\mathbb{Z}}\left(I_{k_1,k_2,\ldots,k_{\ell}}+n\right)\right\},
\end{equation}

\item
\begin{equation}
\left\{d_m=k\right\}=
\left\{\log_{b}x \in\bigcup_{\substack{n\in\mathbb{Z},\\d_1,d_2,...,d_{m-1}}}\left(I_{d_1,d_2,\ldots, d_{m-1},k}+n\right)\right\},
\end{equation}
where $d_1$ varies over $\{1,2,...,b-1\}$, and $d_2,...,d_{k-1}$ vary over $\{0,1,2,...,b-1\}$.
\item 
\begin{equation} \label{eqn:manydigitsshift}
\left\{d_{m_1}=k_1, d_{m_2}=k_2,...,d_{m_L}=k_L\right\}=
\left\{\log_{b}x\in\bigcup_{n\in\mathbb{Z}}\bigcup_{\substack{ d_j, j\in\{1,...,m_L\}, \\j\not=m_s \text{ for }s\in\{1,...,L\},\\d_{m_s}=k_s, \text{ for }s\in\{1,...,L\}}}\left(I_{d_1,d_2,\ldots,d_{m_L}}+n\right)\right\}.
\end{equation}
\end{enumerate}
\end{lem}

\begin{proof}
\begin{enumerate}
\item
\begin{multline*} 
\left\{d_1=k\right\}=\bigcup_{n\in\mathbb{Z}}\left\{kb^n\leq x<(k+1)b^n\right\}=\bigcup_{n\in\mathbb{Z}}\left\{\log_{b}k+n\leq \log_{b}x<\log_{b}(k+1)+n\right\}\\
=\left\{\log_bx\in\bigcup_{n\in\mathbb{Z}}[\log_{b}k+n,\log_{b}(k+1)+n)\right\}.
\end{multline*}
We recognise the right-hand side of \eqref{eqn:1digitcase}.
\item
The computation is analogous to the one for a single digit. We leave the details to the reader. 
\item 
For this case observe that
$$ \left\{d_m=k\right\}=\bigcup_{k_1\in\{1,2,...,b-1\}}\bigcup_{\substack{k_2,...,k_{m-1}\in
\{0,1,2,...,b-1\}}}\left\{d_{1}=k_1, d_{2}=k_2,...,d_{m}=k\right\}.$$

Then the conclusion follows from \eqref{eqn:manyfixeddigits}. 

%\begin{multline*}
   %  \left\{d_m=k\right\}=\\
%\bigcup_{n\in\mathbb{Z}}\bigcup_{d_1\in\{1,2,...,b-1\}}\bigcup_{\substack{d_2,...,d_{m-1}\in\\
%\{0,1,2,...,b-1\}}}
%\left\{b^n
    %\left(
       % \sum_{j=1}^{m-1} d_jb^{1-j}+kb^{1-m}
    %\right)
    %\leq x<b^n
   % \left( \sum_{j=1}^{m-1} d_jb^{1-m}+(k+1)b^{1-m}
   % \right)
%\right\}\\
%=\left\{
    %\log_{b}x\in\bigcup_{\substack{n,\\d_1,d_2,...,d_{k-1}}}
 %   \left[n+\log_{b}
    %    \left(\sum_{j=1}^{m-1} d_jb^{1-j}+kb^{1-m}\right), \right.\right.\\
 %        \left. \left. n+\log_{b}
         %\left( \sum_{j=1}^{k-1} d_jb^{1-j}+(k+1)b^{1-m}
      %   \right)
   % \right)\right\}.
%\end{multline*}

\item 
This case is completely analogous to the previous one if we replace $k$ with $k_L$ and $m$ with $m_L$ only we do not need to take union over the $m_1$-th, $m_2$-th,...,$m_{L-1}$-th digits and instead they are fixed.  We leave the details to the reader.

%Let $k_1, k_2\in \mathbb{N}, \text{ with } k_1<k_2$,  $d_1, d_2\in{\set{0,1,2,...,9}}.$ In a completely analogous way we see that,
%$$\left\{ X_N[k_1]=d_1, X_N[k_2]=d_2\right\}=$$

%\begin{eqnarray*}
%\left\{\log_{10}X_N\in\bigcup_{\substack{n\in\mathbb{Z},\\a_1,a_2,...,a_{k_1-1}\\ a_{k_1}=d_1,\\ a_{k_1+1},..., a_{k_2-1}}}\left[n+\log_{10}\left(\sum_{m=1}^{k-1} a_m10^{1-m}+d_210^{1-k_2}\right),\right. \right. \\
%\left. \left. n+\log_{10}\left( \sum_{m=1}^{k_2-1} a_m10^{1-m}+(d_2+1)10^{1-k_2}\right)\right)\right\}.
%\end{eqnarray*}

\end{enumerate}
\end{proof}

The above lemma illustrates that the events of interest concerning the digits are determined at the logarithmic scale by shifted copies of the same collection of disjoint intervals in $[0,1).$ The next two propositions allow us to estimate the probability of these events via the uniform distribution $\mathcal{U}(0,1)$.

\begin{lem}\label{lemma:poissonsum}
Let $\alpha>1$. Then
\begin{equation}
d_{TV}\left(\alpha\mathcal{N}(0,1)\mod 1,\quad \mathcal{U}(0,1)\right)\leq4\exp\left(-\alpha^22\pi^2\right).
\end{equation}

\end{lem}
\begin{proof}
It suffices to restrict our attention to a union over a set of disjoint intervals. Let $\left\{[a_1,b_1), [a_2,b_2),\ldots,[a_n,b_n)\ldots\right\}$ be a set of disjoint intervals contained in $[0,1]$, with \\$\sum_{n\in\mathbb{N}}(b_n-a_n)=l\leq 1$.
	Set
	$$P=\mathbb{P}\left(\alpha\mathcal{N}(0,1)\in\bigcup_{k\in\mathbb{Z}}\bigcup_{n\in\mathbb{N}}[a_n+k,b_n+k)\right).$$
By countable additivity and after bringing one sum under the integral sign we find
\begin{equation*}
P=\sum_{k\in\mathbb{Z}}\sum_{n=1}^{\infty}\int_{a_n+k}^{b_n+k}\frac{1}{\alpha\sqrt{2\pi}}\exp\left(-\frac{x^2}{2\alpha^2}\right)dx%=\sum_{n=1}^{\infty}\sum_{k\in\mathbb{Z}}\int_{a_n}^{b_n}\frac{1}{\alpha\sqrt{2\pi}}\exp\left(-\frac{(x+k)^2}{2\alpha^2}\right)dx\\
=\sum_{n=1}^{\infty}\int_{a_n}^{b_n}\sum_{k\in\mathbb{Z}}\frac{1}{\alpha\sqrt{2\pi}}\exp\left(-\frac{(x+k)^2}{2\alpha^2}\right)dx.
\end{equation*}
Then, using Poisson's summation formula we obtain
$$
P=\sum_{n=1}^{\infty}\int_{a_n}^{b_n}\sum_{k\in\mathbb{Z}}\exp(-\alpha^22\pi^2k^2)e^{ik2\pi x}dx.
$$
By splitting off the $k=0$ term we then have
\begin{multline*}
P=\sum_{n=1}^{\infty}\int_{a_n}^{b_n}dx+\sum_{n=1}^{\infty}\sum_{k\not=0}\exp(-\alpha^22\pi^2k^2)\frac{e^{2\pi i b_n k}-e^{2\pi i a_n k}}{2\pi i k}\\=l+\sum_{k\not=0}\exp(-\alpha^22\pi^2k^2)\sum_{n=1}^{\infty}\frac{e^{2\pi i b_n k}-e^{2\pi i a_n k}}{2\pi i k}.
\end{multline*}
To prove the statement it remains to bound the series on the right-hand side. 
From the observation that $|e^{ib}-e^{ia}|\leq |b-a|$ it follows that
\begin{equation*}
\left|\sum_{n=1}^{\infty}\frac{e^{2\pi i b_n k}-e^{2\pi i a_n k}}{2\pi i k}\right|\leq\sum_{n=1}^{\infty}\frac{2\pi k b_n-2\pi k a_n}{2\pi k}=\sum_{n=1}^\infty(b_n-a_n)=l\leq1.
\end{equation*}
Thus,
\begin{multline*}
\left|P-l\right|\leq 2\sum_{k=1}^{\infty}\exp(-\alpha^22\pi^2k^2)
\leq 2\exp(-\alpha^22\pi^2)\sum_{k=0}^{\infty}\exp(-k\alpha^22\pi^2)=\frac{2\exp(-\alpha^22\pi^2)}{1-\exp(-\alpha^22\pi^2)}%\leq 4\exp(-\alpha^22\pi^2),
\end{multline*}
and this completes the proof.
\end{proof}
Note that  Lemma \ref{lemma:poissonsum} demonstrates how close the normal distribution $\mathcal{N}(0,\alpha)$ modulo $1$ is  to a uniform random variable on $[0,1)$ as $\alpha \to \infty$.  We then proceed to use Theorem \ref{thm:totalvariation} to show that $\log \left|D_{N,\beta}(\theta)\right| \mod 1$  is close to $\mathcal{U}(0,1)$. The next proposition allows us to do that.

\begin{prop}[Total variation convergence modulo $1$]\label{prop:cortotalvariation}
Let $(\alpha_N)_{N\in\mathbb{N}},(\gamma_N)_{N\in\mathbb{N}}$ be sequences of positive real numbers such that
$$\alpha_N\rightarrow\infty,\quad\text{as}\quad N\rightarrow\infty, \quad\gamma_{N}\rightarrow0,\quad\text{as}\quad N\rightarrow\infty.$$ 
 Suppose $X_N$ is a sequence of positive random variables such that 
$$d_{TV}(\log X_N, \alpha_{N}\mathcal{N}(0,1))=\mathcal{O}(\gamma_N).$$
Then for any fixed $b>0$,
\begin{equation}\label{eqn:tvgeneral}
d_{TV}\left(\log_b X_N\mod 1,\quad \mathcal{U}(0,1)\right)=\mathcal{O}\left(\gamma_N+\exp\left(-\frac{2\pi\alpha_N^2}{\log^2b}\right)\right).
\end{equation}
%In particular, 
%\begin{equation}\label{eqn:tovareqn}
%d_{TV}\left(\log_b\left|D_{N,\beta}(\theta)\right| \mod 1,\quad \mathcal{U}(0,1)\right)=o(T_N^{-3+\varepsilon}) \text{ as } N\rightarrow \infty.
%\end{equation}
\end{prop}
\begin{proof}
For a sequence $X_N$ as in the setting,
\begin{multline*}
d_{TV}\left(\log_b X_N \mod 1,\quad \mathcal{U}(0,1)\right)\\
\leq d_{TV}\left(\log_b X_N, \quad\log_b \alpha_N\mathcal{N}(0,1)\right)\\+d_{TV}\left(\log_b \alpha_N\mathcal{N}(0,1)\mod 1,\quad\mathcal{U}(0,1)\right)\\=\mathcal{O}\left(\gamma_N+\exp\left(-\frac{2\pi\alpha_N^2}{\log^2b}\right)\right),
\end{multline*}
by an application of Lemma \ref{lemma:poissonsum}. Hence we proved \eqref{eqn:tvgeneral}. %To deduce \eqref{eqn:tovareqn} we plug in the the rate of convergence $\gamma_N=T_N^{-3+\varepsilon}$ proved in Theorem \ref{thm:totalvariation}.

\end{proof}

%\begin{rmk}
%The rate of convergence in Proposition \ref{prop:cortotalvariation} comes from the rate of convergence in total variation from Lemma \ref{thm:totalvariation}, since the bound in Lemma \ref{lemma:poissonsum} decays exponentially. 
%\end{rmk}
%Proposition \ref{prop:cortotalvariation} together with Theorem \ref{thm:totalvariation} and \eqref{eqn:variancebound} imply Theorems \ref{thm1} and \ref{thm2}. %In fact Proposition \ref{prop:cortotalvariation} implies analogous results about the digits of a sequence of positive random variables as long as its conditions are satisfied. Avoiding introducing new notation, we restrict ourselves to proving our main theorems.
Proposition \ref{prop:cortotalvariation} is the core part of Theorem \ref{thm:sufficientconditions}. Since the conditions of Proposition \ref{prop:cortotalvariation} are satisfied in the setting of $\left|D_{N,\beta}(\theta)\right|$, deducing our main results is a matter of understanding the lengths of the intervals determining the digits of a random variable as described in Lemma \ref{lemma:intervalsdescription}. In the case of Theorem \ref{thm1}, as we see in \eqref{eqn:manyfixeddigits}, this is the length of the interval $\left[\log_{b}\left(\sum_{j=1}^{\ell} k_jb^{1-j}\right), 
\log_{b}\left( \sum_{j=1}^{\ell} k_jb^{1-j}+b^{1-\ell}\right)\right)$ for digits $k_1\in\left\{1,2,...,b-1\right\}$, $k_2,\ldots,k_{\ell}\in\{0,1,\ldots,b-1\}$ which is $\log_b\left(1+\frac{1}{\sum_{j=1}^{\ell} k_jb^{l-j}}\right)$. To understand the length of the intervals for Theorem \ref{thm2} we prove the next lemma.

\begin{lem}[Length of intervals]\label{lemma:lengthofintervals}
Let $m_1,m_2,...,m_L\in\mathbb{N}$, with $2\leq m_1< m_2<...<m_L$ and let $k_1,k_2,...,k_L\in\{0,1,\ldots, b-1\}$ be digital values.  Consider the collection of $(b-1)b^{m_L-L-1}$ disjoint intervals in $[0,1]$, indexed by the values $k_1,k_2,\ldots,k_L\in\{0,1,\ldots,b-1\}$ of the digits in positions $m_1,m_2,\ldots,m_L$ given by
 \begin{equation}\label{eqn:Adefn}
 \mathbb{A}^{m_1,m_2,...,m_L}_{k_1,k_2,\ldots,k_L}=\bigcup_{\substack{d_1\in\{1,2,\ldots,b-1\},\\d_2,...,d_{m_L}\in\{0,1,2,...,b-1\}, \\ d_{m_1}=k_1, d_{m_2}=k_2, ..., d_{m_L}=k_L}}I_{d_1,d_2,\ldots,d_{m_L}}.
 \end{equation}
Then the Lebesgue measure of $\mathbb{A}^{m_1,m_2,...,m_L}_{k_1,k_2,\ldots,k_L}$, denoted by $\left|\mathbb{A}^{m_1,m_2,...,m_L}_{k_1,k_2,\ldots,k_L}\right|$, satisfies the following bound
\begin{equation}
\left|\left|\mathbb{A}^{m_1,m_2,...,m_L}_{k_1,k_2,\ldots,k_L}\right|-\frac{1}{b^L}\right|\leq \log_{b}\left(1+\frac{1}{b^{m_1-1}}\right)<\frac{1}{b^{m_1-1}}.
\end{equation}
\end{lem}

\begin{proof}[Proof of Lemma \ref{lemma:lengthofintervals}]
By construction, $x\in{ [1,b)}$ has first $m_L$ digits $d_1,d_2,...,d_{m_L}$ if and only if $\log_bx\in I_{d_1,d_2,\ldots,d_{m_L}}$, recall \eqref{eqn:manyfixeddigits}. Hence the set of such intervals (varied over $d_1,\ldots, d_{m_L}$ but with $m_L$ fixed) is a partition of $[0,1)$ and bijects with decimal numbers in $[1,b)$ with representation base $b$, $\sum_{j=1}^{m_L}d_jb^{1-j}$, or $m_L$ tuples $(d_1,d_2,\ldots, d_{m_L})\in\{1,2,\ldots,b-1\}\times\{0,1,\ldots,b-1\}^{m_L-1}$. We will use this correspondence throughout the proof. Observe further that, again by construction, $x$ lies in the interval ${[1,b)}$ and has $m_1$-th digit $k_1$, $m_2$-th digit $k_2$,..., $m_L$-th digit $k_L$ if and only if $\log_b(x) \in \mathbb{A}^{m_1,m_2,...,m_L}_{k_1,k_2,\ldots,k_L}$, recall \eqref{eqn:manydigitsshift}. It follows that $\mathbb{A}^{m_1,m_2,...,m_L}_{k_1,k_2,\ldots,k_L}$ for $k_1,\ldots,k_{L}\in \left\{0,1,\ldots,b-1\right\}$ form  a partition of the interval $[0,1)$. Hence
\begin{equation}\label{eqn:addto1}
\sum_{k_1,k_2,\ldots,k_L}\left|\mathbb{A}^{m_1,m_2,...,m_L}_{k_1,k_2,\ldots,k_L}\right|=1.
\end{equation}
%Denote the set consisting of all intervals in $\mathbb{A}^{m_1,m_2,...,m_L}_{k_1,k_2,\ldots,k_L}$, $\bigsqcup_{\mathbb{A}^{m_1,m_2,...,m_L}_{k_1,k_2,\ldots,k_L}}I\subset [0,1]$. 
Now for any fixed $(k_1,k_2,...,k_L)$, each individual interval $I_{d_1,d_2,\ldots,d_{m_L}}$ in $\mathbb{A}^{m_1,m_2,...,m_L}_{k_1,k_2,\ldots, k_L}$ indexed by $d_1,d_2,...,d_{m_L}$ has length $\log_b\left(1+\frac{1}{\sum_{m=1}^{m_L} d_mb^{m_L-m}}\right)$. Hence,
\begin{equation}\label{eqn:descriptionsum}
\left|\mathbb{A}^{m_1,m_2,...,m_L}_{k_1,k_2,\ldots,k_L}\right|=\sum_{\substack{ d_m, m\in\{1,...,m_L\},\\d_{m_j}=k_j, \text{ for }j\in\{1,...,L\}}}\log_b\left(1+\frac{1}{\sum_{m=1}^{m_L} d_mb^{m_L-m}}\right).
\end{equation}
We proceed to argue that all $\left|\mathbb{A}^{m_1,m_2,...,m_L}_{k_1,k_2,\ldots,k_L}\right|$ are of the same order of magnitude as $m_1\rightarrow~\infty$. Let $(k_1,k_2,...,k_L),(k_1',k_2',...,k_L')\in\left\{0,1,2,...,b-1\right\}^L$ be two distinct $L$-tuples such that $$\sum_{j=1}^{L} k_jb^{1-m_j}<\sum_{j=1}^{L} k_j'b^{1-m_j},$$ then comparing the sums of the form of \eqref{eqn:descriptionsum} termwise and using the monotonicity of $y \mapsto \log\left(1+\frac{1}{y}\right)$ we see that
$$\left|\mathbb{A}^{m_1,m_2,...,m_L}_{k_1,k_2,\ldots,k_L}\right|>\left|\mathbb{A}_{k_1',k_2',...,k_L'}^{m_1,m_2,\ldots,m_L}\right|.$$
In particular, if we denote $\mathbb{A}_{0,0,...,0}^{m_1,m_2,\ldots,m_L}=\mathbb{A}_{\textbf{0}}$, $\mathbb{A}_{b-1,b-1,...,b-1}^{m_1,m_2,\ldots,m_L}=\mathbb{A}_{\textbf{b-1}}$, then for any $(k_1,k_2,...,k_L)\in \left\{0,1,2,...,b-1\right\}^L$,
\begin{equation}\label{eqn:monotonicitydeduction}
|\mathbb{A}_{\textbf{0}}|\geq \left|\mathbb{A}^{m_1,m_2,...,m_L}_{k_1,k_2,\ldots,k_L}\right|\geq|\mathbb{A}_{\textbf{b-1}}|.
\end{equation}
 %The second property follows from the fact that $\log_B{1+1/x}$ is decreasing in $x$, and for fixed digits $a_1,...,a_{K_L}$ the length of the interval $\left[\log_{B}\left(\sum_{m=1}^{k_L} a_mB^{1-m}\right), \log_{B}\left( \sum_{m=1}^{k_L} a_mB^{1-m}+B^{1-k_L}\right)\right)$ is \\$\log\left(1+\frac{1}{\sum_{m=1}^{k_L} a_mB^{k_L-m}}\right)$.  
We now show $$|\mathbb{A}_{\textbf{b-1}}|>|\mathbb{A}_{\textbf{0}}|-\log_b\left(1+\frac{1}{b^{m_1-1}}\right).$$ 
We compare the terms as in \eqref{eqn:descriptionsum} for $\left|\mathbb{A}_{\textbf{b-1}}\right|$ and $\left|\mathbb{A}_{\textbf{0}}\right|$  using the following "shift" operation. Given an index  $(d_1,d_2,\ldots,d_{m_L})$  for  interval in  $\mathbb{A}_{\textbf{b-1}}$ we construct a new index $(d_1', d_2',\ldots, d_{m_L}')$ for an interval in $\mathbb{A}_{\textbf{0}}$ by  replacing the $m_1$-th, $m_2$-th,...,$m_L$-th digit (which are $b-1$) with $0$ and raising the value of the $(m_1-1)$-th digit by $+1$. More precisely,  the first $m_1-1$ digits $(d_1',d_2',\ldots, d_{m_1-1}')$ are the digits of the integer $1+\sum_{j=1}^{m_1-1}d_jb^{m_1-1-j}$. We will exclude the special  cases with $d_1=d_2=\ldots=d_{m_1-1}=b-1$ (which would produce an additional digit).~\footnote{Three examples for this operation base $10$ for $L=2$, $m_1=3$, $m_2=6$ are: $129789\mapsto 130780$, $919449\mapsto 920440$, $299569\mapsto 300560$. We exclude the special  cases  $999229$ as it would map to $1000220$ which has $7$ digits instead of~$6$.} Observe also that by construction $\sum_{j=1}^{m_L}d'_jb^{m_L-j}>\sum_{j=1}^{m_L}d_jb^{m_L-j}$ and hence
\begin{equation*}\label{eqn:comparingtermsintervals}
\log_b\left(1+\frac{1}{\sum_{m=1}^{m_L} d_mb^{m_L-m}}\right)>\log_b\left(1+\frac{1}{\sum_{m=1}^{m_L} d_m'b^{m_L-m}}\right).
\end{equation*}
Comparing the length of intervals along this operation we see that $$\left|\mathbb{A}_{\textbf{b-1}}\setminus\left\{\substack{I_{d_1,d_2,\ldots, d_{m_L}}:\\ d_1=d_2=\ldots=d_{m_1-1}=b-1}\right\}\right|>\left|\mathbb{A}_{\textbf{0}}\setminus\left\{\substack{I_{d_1,d_2,\ldots, d_{m_L}}:\\d_1=1,d_2=\ldots=d_{m_1-1}=0}\right\}\right|.$$
In particular
$$|\mathbb{A}_{\textbf{b-1}}|>|\mathbb{A}_{\textbf{0}}|-\left|\bigcup_{\left\{\substack{I_{d_1,d_2,\ldots, d_{m_L}}\in \mathbb{A}_{\textbf{0}}:\\d_1=1,\quad d_2=\ldots=d_{m_1}=0}\right\}}I_{d_1,d_2,...,d_{m_L}}\right|>|\mathbb{A}_{\textbf{0}}|-\left|\left\{\log_b(x):\substack{x\in[1,b),\quad x \text{ has leading digits}\\ d_1=1,\quad d_2=\ldots=d_{m_1}=0 }\right\}\right|.$$
The last term on the right-hand side is the measure of an interval of the form $\left[0,\log_b\left(1+b^{1-m_1}\right)\right)$, and so
\begin{equation}\label{eqn:nomonotonicitydeduction}
|\mathbb{A}_{\textbf{b-1}}|>|\mathbb{A}_{\textbf{0}}|-\log_b\left(1+\frac{1}{b^{m_1-1}}\right).
\end{equation}
Hence, combining \eqref{eqn:monotonicitydeduction} and \eqref{eqn:nomonotonicitydeduction} we see that for any $(k_1,k_2,...,k_L)\in\{0,1,\ldots, b-1\}^{L}$, 
\begin{equation}\label{eqn:boundfortuple}
|\mathbb{A}_{\textbf{0}}|\geq\left|\mathbb{A}^{m_1,m_2,...,m_L}_{k_1,k_2,\ldots,k_L}\right|>|\mathbb{A}_{\textbf{0}}|-\log_b(1+\frac{1}{b^{m_1-1}}).
\end{equation}
 Summing \eqref{eqn:boundfortuple} over all possible $(k_1,\ldots, k_L)\in\{0,1,\ldots,b-1\}^{L}$, due to \eqref{eqn:addto1} we see 
$$b^L|\mathbb{A}_{\textbf{0}}|>1>b^L|\mathbb{A}_{\textbf{0}}|-b^L\log_{b}\left(1+\frac{1}{b^{m_1-1}}\right),\text{ and so}$$
 \begin{equation}\label{eqn:bounding1bl}
 |\mathbb{A}_{\textbf{0}}|>\frac{1}{b^L}>\mathbb{A}_{\textbf{0}}-\log_{b}\left(1+\frac{1}{b^{m_1-1}}\right).
 \end{equation}
Hence, because of \eqref{eqn:boundfortuple} and \eqref{eqn:bounding1bl} it follows that both  $\left|\mathbb{A}^{m_1,m_2,...,m_L}_{k_1,k_2,\ldots,k_L}\right|,\frac{1}{b^L}$ lie in the interval $\left(|\mathbb{A}_{\textbf{0}}|-\log_b(1+\frac{1}{b^{m_1-1}}), \quad|\mathbb{A}_{\textbf{0}}|\right]$, and so
$$\left||\mathbb{A}_{k_1,k_2,...,k_L}|-\frac{1}{b^L}\right|< \log_{b}\left(1+\frac{1}{b^{m_1-1}}\right) \text{ for any } (k_1,k_2,\ldots, k_L)\in\{0,1,\ldots, b-1\}^L.$$
This finishes the proof.
%just comparing termwise, terms with the sane choice of $a_1,...,a_{k-1}$, since $|\mathbb{A}_{d_1}|>|\mathbb{A}_{d_2}|$
\end{proof}
A result such as Lemma \ref{lemma:lengthofintervals} quantifies the change of behavior of higher digits for Benford's law and is necessary to study the problem we investigate in Theorem \ref{thm2}. The skewed measure towards the small digits of a random variable with uniformly distributed logarithm becomes uniform exponentially fast in the position of the digit.
We can now combine Lemma \ref{lemma:intervalsdescription}, Proposition \ref{prop:cortotalvariation} and Lemma \ref{lemma:lengthofintervals} to prove Theorem \ref{thm:sufficientconditions}. It then implies Theorems \ref{thm1} and \ref{thm2}, via Theorem \ref{thm:totalvariation} and \eqref{eqn:variancebound} in Lemma \ref{lemma:near0}.
\subsection{Proofs of Theorems \ref{thm1}, \ref{thm2}, \ref{thm:sufficientconditions}} \label{proofsthms}
\begin{proof}[Proof of Theorem \ref{thm:sufficientconditions}]
We first prove \eqref{eqn:thm1eqn}. By Lemma \ref{lemma:intervalsdescription}, \eqref{eqn:manyfixeddigits},
\begin{equation*}\mathbb{P}( d_{1,N}=k_1,d_{2,N}=k_2,...,d_{\ell,N}=k_\ell)
=\mathbb{P}\left(\bigcup_{M\in\mathbb{Z}}\left\{\log_{b}X_N\in\left(I_{k_1,k_2,\ldots,k_{\ell}}+M\right)\right\}\right).
\end{equation*}
By Proposition \ref{prop:cortotalvariation} it follows that 
\begin{multline*}
\mathbb{P}( d_{1,N}=k_1,d_{2,N}=k_2,...,d_{\ell,N}=k_\ell)=\mathbb{P}\left(\mathcal{U}(0,1)\in I_{k_1,k_2,\ldots,k_{\ell}}\right)+\mathcal{O}\left(\gamma_N+\exp\left(-\frac{2\pi\alpha_N^2}{\log^2b}\right)\right)\\
    =\log_b\left(\sum_{n=1}^{\ell}k_nb^{-n+1}+b^{-\ell+1}\right)-\log_{b}\left(\sum_{n=1}^{\ell}k_nb^{-n+1}\right)+\mathcal{O}\left(\gamma_N+\exp\left(-\frac{2\pi\alpha_N^2}{\log^2b}\right)\right), \text{ as } N\rightarrow\infty.
\end{multline*}    
By writing the difference of logarithms as one we obtain \eqref{eqn:thm1eqn}. 
%\end{proof}
%\subsection{Proof of Theorem \ref{thm2}}\label{proofThm2}
%The proof of Theorem \ref{thm2} is analogous to that of Theorem \ref{thm1}. We first rewrite the probability of the event we study as the probability that $\log|D_{N,\beta}(\theta)|$ lies in a disjoint union of intervals as in Proposition \ref{prop:cortotalvariation}. However this time we do not have an exact expression for the length of these intervals, and we need an estimate of it. Lemma \ref{lemma:lengthofintervals} is such an estimate. 
%\begin{proof}[Proof of Theorem \ref{thm2}]

The proof of \eqref{eqn:thm2eqn} is analogous. We stay consistent with the notation of Lemma \ref{lemma:lengthofintervals}, $\mathbb{A}_{k_1,k_2,\ldots,k_L}^{m_1,m_2,\ldots,m_L}$ defined from \eqref{eqn:Adefn} where the corresponding positions of the digits are \\$m_1(N), m_2(N),...,m_L(N)$ as in the setting of the theorem. Using \eqref{eqn:manydigitsshift} from Lemma \ref{lemma:intervalsdescription} we see 
$$\mathbb{P}(d_{m_1(N),N}=k_1, \ldots, d_{m_L(N),N}=k_L)=\mathbb{P}\left(\log_{b}X_N\in\bigcup_{M\in\mathbb{Z}}\left(\mathbb{A}_{k_1,k_2,...,k_L}^{m_1,m_2,\ldots, m_L}+M\right)\right).$$
Proposition \ref{prop:cortotalvariation} then implies that
$$
\mathbb{P}\left(\log_{b}X_N\in\bigcup_{M\in\mathbb{Z}}\left(\mathbb{A}_{k_1,k_2,...,k_L}^{m_1,m_2,\ldots, m_L}+M\right)\right)=\left|\mathbb{A}_{k_1,k_2,...,k_L}^{m_1,m_2,\ldots,m_L}\right|+\mathcal{O}\left(\gamma_N+\exp\left(-\frac{2\pi\alpha_N^2}{\log^2b}\right)\right),
$$
as $N \rightarrow \infty$ and the convergence is independent of the choice of digits. Finally Lemma \ref{lemma:lengthofintervals} allows us to conclude
\begin{multline*}
\mathbb{P}\left(\log_{b}X_N\in\bigcup_{M\in\mathbb{Z}}\left(\mathbb{A}_{k_1,k_2,...,k_L}^{m_1,m_2,\ldots, m_L}+M\right)\right)=\frac{1}{b^L}+\mathcal{O}\left(\gamma_N+\exp\left(-\frac{2\pi\alpha_N^2}{\log^2b}\right)+\frac{1}{b^{m_1(N)}}\right),\\ \text{ as } N\rightarrow\infty.
\end{multline*}
This concludes the second part of Theorem \ref{thm:sufficientconditions}, \eqref{eqn:thm2eqn}.
\end{proof}

\begin{proof}[Proof of Theorem \ref{thm1}]
By Theorem \ref{thm:totalvariation} and Lemma \ref{lemma:near0}, the conditions of Theorem \ref{thm:sufficientconditions} are satisfied for $\log\left|D_{N,\beta}(\theta)\right|$ with $\gamma_N=(\log N)^{-3/2+\varepsilon}$ and $\alpha_N=T_N$. Hence by Theorem \ref{thm:sufficientconditions},
\begin{multline*}
  \left|\mathbb{P}( d_{1,N}(\theta)=k_1,d_{2,N}(\theta)=k_2,...,d_{\ell,N}(\theta)=k_\ell)-\log_{b}\left(1+\frac{1}{\sum_{j=1}^{\ell}k_j b^{\ell-j}}\right)\right|\\=\mathcal O\left((\log N)^{-3/2+\varepsilon}+\exp\left(-\frac{2\pi T_N^2}{\log^2b}\right)\right)=\mathcal{O}\left((\log N)^{-3/2+\varepsilon}\right),\quad\text{ as}\quad N\rightarrow\infty.
\end{multline*}
\end{proof}

\begin{proof}[Proof of Theorem \ref{thm2}]
The proof of Theorem \ref{thm2} is an analogous application of Theorem \ref{thm:sufficientconditions} to the context of $\left|D_{N,\beta}(\theta)\right|$, together with Theorem \ref{thm:totalvariation} and Lemma \ref{lemma:near0}. 
\end{proof}

\appendix
\section{A counterexample} \label{sec:counter}

Let $(Y_n)_{n\in \mathbb{N}}$ be sequence of random variables such that the  cumulants $(C_k^{(n)})_{n\in\mathbb{N}}$ satisfy 
\begin{itemize}
\item $C_1^{(n)}=0, C_2^{(n)}\rightarrow\infty$ as $n\rightarrow\infty$ and
\item $C_k^{(n)}\leq const (k-1)!$ for $k\geq 3.$
\end{itemize}
Then it is classical that $Y_n /\sqrt{C^{(n)}_2}$ converges to a standard normal in distribution. In a heuristic derivation Keating and Snaith \cite{KeatingSnaith} showed that these bounds on the cumulants may give the following asymptotics of the density $\rho_n(x)$ of $Y_n /\sqrt{C^{(n)}_2}$
\begin{equation}\label{eq:wrongKeatingSnaith}
\rho_n(x)=\frac{1}{\sqrt{2 \pi}}e^{-x^2/2} + \mathcal O\left(\left(C_2^{(n)}\right)^{-3/2}\right), 
\end{equation}
as $n \to \infty$ (and cited as Theorem A.1 in  \cite{KantMill} in the context of the CUE). However, in the derivation of Keating and Snaith a further justification of a change of limit and integral is needed (the expansion in the integrand at the right-hand side of (50) in \cite{KeatingSnaith} is not valid on the entire  interval of integration). Such a justification is present in Theorem \ref{thm:boundingsectionsummary}.  Without this justification the result may not be true. Indeed,  it is possible to find  a sequence of random variables $Y_n$ for which the bound on the cumulants hold, but \eqref{eq:wrongKeatingSnaith} fails.

For this counterexample consider $X_n=\sum_{k=1}^n\xi_k$, where $(\xi_k)_{k\in\mathbb{N}}$ be a sequence of i.i.d. random variables with $\mathbb{P}(\xi_1=1)=\mathbb{P}(\xi_1=-1)=1/2$. The characteristic function of $\xi_1$ is $\cos(t)$, and hence $X_N=\sum_{k=1}^N\xi_k$ has characteristic function $\phi_N(t)=\cos(t)^N$. Now, using an infinite product representation of cosine, we have the following identity:
$$\log\left(\cos(t)\right)=-\sum_{m\geq1}\frac{(4^m-1)\zeta{(2m)}}{m\pi^{2m}}t^{2m}.$$
From this expression we can read off the cumulants of $\xi_1$ and deduce that the cumulants of $X_N$ are: 
$$C_j^{(N)}=\begin{cases}
0\text{, if $j$ is odd,} \\
\frac{N2(j-1)!(2^j-1)\zeta(j)}{\pi^{j}} \text{, if $j$ is even.}
\end{cases}$$
In particular $C_2^{(N)}=N$. Hence the cumulants of $Y_N=X_N/\left(C_2^{(N)}\right)^{1/4}$ satisfy the conditions described in the beginning.  However, pointwise convergence of the p.d.f. of $Y_N/\left(C_2^{(N)}\right)^{1/4}$ to $\frac{1}{\sqrt{2\pi}}e^{-x^2/2}$ cannot occur as it is discrete. 

\section{Convergence in Kolmogorov distance}

Since the Kolmogorov distance is trivially bounded by the total variation distance, Theorem \ref{thm:totalvariation} also implies a CLT for $\frac{\log |D_{N,\beta}(\theta)|}{\sqrt{\var(\log |D_{N,\beta}(\theta)|)}}$ with convergence in Kolmogorov distance of order  $\mathcal{O}\left({T_N^{-3+\varepsilon}}\right)$ for arbitrary $\varepsilon>0$.  With a little extra work, this can be improved to $\mathcal{O}\left({T_N^{-3}}\right)$, as we will show now. Note that this order of convergence has also been proved for $\beta=2$ in \cite{BHNY, Ashkan}.
% Even though CLTs in total variation for this object have not been studied before we mention various results on CLTs for specific $C_\beta$ ensembles in Kolmogorov distance. Such a result is proved for the Circular $\beta$ Ensemble  in \cite{DalBorgo} where a rate of $O\left(\frac{1}{T_N}\right)$ is obtained in \cite{DalBorgo} and for  3 specific $C_\beta$ ensembles at rate $O\left(\frac{1}{T_N^3}\right)$ in \cite{Ashkan} with the tools of Mod-Gaussian convergence that they develop. We describe a way to apply our results and a classical Lemma to obtain a rate of convergence in Kolmogorov distance of $\frac{1}{T_N^3}$.
%The following Lemma is a classical result usually used to prove Berry-Esseen's Theorem, see e.g. Lemma 1 on page 538 in \cite{Feller}.
%\begin{lem}
%Let $F$ be a probability distribution and G a differentiable function such that $G(-\infty)=0$, $G(\infty)=1$, and $|G'(x)|\leq m<\infty,\forall x\in\mathbb{R}$. Let:

We will need the following classical lemma.
\begin{lem}\label{lemma:petrov}
Let $F(x)$ be a non-decreasing function, $G(x)$ a differentiable function of bounded variation with $\sup_{x\in\mathbb{R}}|G'(x)|\leq C$ for some constant $C>0$. Assume further that $F(-\infty)=G(-\infty)$, $F(+\infty)=G(+\infty)$. Let
$$f(t)=\int_{\mathbb{R}}e^{itx}dF(x),$$
$$g(t)=\int_{\mathbb{R}}e^{itx}dG(x).$$
Let $T>0$. Then for every $b>\frac{1}{2\pi}$, there exists $r(b)>0$ s.t.
$$
\sup_{x\in\mathbb{R}}|F(x)-G(x)|\leq b\int_{-T}^{T}\left|\frac{f(t)-g(t)}{t}\right|dt+r(b)\frac{C}{T}.
$$
\end{lem}
\begin{proof}
See Theorem 2 in \cite{Petrov}.
\end{proof}
In particular this lemma can be applied for $G(x)=\frac{1}{\sqrt{2\pi}}e^{\frac{-x^2}{2}}$, and for $\rho_{N}(x)$ the p.d.f. of $\frac{\log|D_{N,\beta}(\theta)|}{T_N}$, $F(x)=\int_{-\infty}^{x}\rho_{N}(y)dy$. Then choosing $b=\pi$, $T=T_N^{3}$ we prove the following result.
\begin{prop}[CLT in Kolmogorov distance]
$$\sup_{x\in\mathbb{R}}\left|\mathbb{P}\left(\frac{\log|D_{N,\beta}(\theta)|}{T_N}\leq x\right)-\int_{-\infty}^{x}\frac{1}{\sqrt{2\pi}}e^{-y^2/2}dy\right|=\mathcal{O}\left(\frac{1}{T_N^3}\right).$$
\end{prop}
\begin{proof}
By Lemma \ref{lemma:petrov},
$$\sup_{x\in\mathbb{R}}\left|\mathbb{P}\left(\frac{\log|D_{N,\beta}(\theta)|}{T_N}\leq x\right)-\int_{-\infty}^{x}\frac{1}{\sqrt{2\pi}}e^{-y^2/2}dy\right|\leq \pi\int_{-T_N^3}^{T_N^3}\left|\frac{\psi_{N,\beta}\left(\frac{t}{T_N}\right)-e^{t^2/2}}{t}\right|dt+r(\pi)\frac{C}{T_N^3}.$$
The second summand on the right-hand side is $O\left(\frac{1}{T_N^3}\right)$. It remains to estimate the integral:
\begin{multline*}
\int_{-T_N^3}^{T_N^3}\left|\frac{\psi_{N,\beta}\left(\frac{t}{T_N}\right)-e^{t^2/2}}{t}\right|dt\leq\int_{-T_N}^{T_N}\left|\frac{\psi_{N,\beta}\left(\frac{t}{T_N}\right)-e^{t^2/2}}{t}\right|dt+\int_{T_N<|t|<T_N^3}\left|\frac{\psi_{N,\beta}\left(\frac{t}{T_N}\right)-e^{t^2/2}}{t}\right|dt \\
\leq\int_{-T_N}^{T_N}\left|\frac{\psi_{N,\beta}\left(\frac{t}{T_N}\right)-e^{t^2/2}}{t}\right|dt+\int_{T_N<|t|<T_N^3}\left|\frac{\psi_{N,\beta}\left(\frac{t}{T_N}\right)}{T_N}\right|dt+\int_{T_N<|t|<T_N^3}\left|\frac{e^{-t^2/2}}{T_N}\right|dt.
\end{multline*}
We estimate the first integral using Lemma \ref{lemma:near0}, the second integral using Proposition \ref{proposition:smallt} and the third integral using the estimate $\int_{A}^{\infty}e^{-x^2/2}dx<\frac{e^{-A^2/2}}{A\sqrt{2\pi}}$ for $A>0$. We see
\begin{multline*}
\int_{-T_N^3}^{T_N^3}\left|\frac{\psi_{N,\beta}\left(\frac{t}{T_N}\right)-e^{t^2/2}}{t}\right|dt\\
\leq\int_{-T_N}^{T_N}\frac{c_1}{T_N^3}e^{-t^2/2}t^2dt+\frac{2}{T_N}\int_{T_N}^{\infty}\exp\left(-c_2\frac{t^2}{2}\right)dt+\frac{1}{T_N^2\sqrt{2\pi}}\exp\left(-\frac{T_N^2}{2}\right)\\
\leq \frac{c_1\sqrt{2\pi}}{T_N^3}+\frac{2}{T_N^2\sqrt{c_22\pi}}\exp\left(\frac{-c_2T_N^2}{2}\right)+\frac{1}{T_N^2\sqrt{2\pi}}\exp\left(-\frac{T_N^2}{2}\right)=\mathcal{O}\left(\frac{1}{T_N^3}\right).
\end{multline*}
The proof is complete.
\end{proof}

\end{document}